\documentclass[a4paper]{article}

\usepackage{amsmath,amsthm,amsfonts,amssymb}
\usepackage{tikz,tikz-cd,tikz-3dplot}
\usetikzlibrary{positioning}
\usepackage{geometry}
\usepackage{multirow}
\usepackage{graphicx}
\usepackage{caption}
\usepackage{subfigure}
\usepackage{float}
\usepackage{booktabs}
\usepackage{tabularx}
\usepgflibrary{shapes.misc}
\usepackage{cite}

\DeclareMathOperator{\Aut}{Aut}

\DeclareMathOperator{\Hom}{Hom}

\DeclareMathOperator{\Pic}{Pic}

\DeclareMathOperator{\Sym}{Sym}

\DeclareMathOperator{\conj}{conj}
\DeclareMathOperator{\Per}{Per}
\DeclareMathOperator{\Real}{Re}

\DeclareMathOperator{\sign}{sign}
\DeclareMathOperator{\adj}{adj}
\DeclareMathOperator{\rank}{rank}
\DeclareMathOperator{\Mat}{Mat}

\theoremstyle{plain}
\newtheorem{thm}{Theorem}[section]
\newtheorem{lem}[thm]{Lemma}
\newtheorem{prop}[thm]{Proposition}
\newtheorem{cor}[thm]{Corollary}

\theoremstyle{definition}
\newtheorem{dfn}[thm]{Definition}
\newtheorem{ex}[thm]{Example}
\newtheorem{rmk}[thm]{Remark}
\newtheorem*{ack}{Acknowledgements}

\begin{document}

\colorlet{LightGray}{gray!30}

\tikzstyle{circ}=[draw,circle,inner sep=1mm]
\tikzstyle{rect}=[draw,rectangle,minimum size=2mm]
\tikzstyle{cross}=[draw,cross out]
\tikzstyle{greycirc}=[draw,circle,fill=black!15,inner sep=1mm]
\tikzstyle{blackcirc}=[draw,circle,fill=black,inner sep=1mm]
\tikzstyle{c}=[draw,circle]

\tikzset{vert/.style={path picture={ 
      \draw[black]
       (path picture bounding box.north) -- (path picture bounding box.south) ;       
       }}}       
       
\tikzset{oplus/.style={path picture={ 
      \draw[black]
       (path picture bounding box.south) -- (path picture bounding box.north) 
       (path picture bounding box.west) -- (path picture bounding box.east);
      }}} 

\title{Hyperbolic Geometry and Moduli
of Real Curves of Genus Three}
\author{Gert Heckman and Sander Rieken}

\maketitle

\abstract{The moduli space of smooth real plane quartic curves consists of six connected components. We prove that each of these components admits a real hyperbolic structure. These connected components correspond to the six real forms of a certain hyperbolic lattice over the Gaussian integers. We will study this Gaussian lattice in detail. For the connected component that corresponds to maximal real quartic curves we obtain a more explicit description. We construct a Coxeter diagram that encodes the geometry of this component.}

\section{Introduction}


Recently there has been a great deal of progress in the construction of period maps from moduli spaces to ball quotients. This allows for a new approach to the study of questions of reality for these moduli spaces. The main example of this in the literature is the work of Allcock, Carlson and Toledo on the moduli space of cubic surfaces. In \cite{ACTcubiccomplex} they construct a period map from this moduli space to a ball quotient of dimension four. The question of reality for this period map is studied in \cite{ACTrealcubic}. One of the five connected components of this real moduli space, the one where all $27$ lines on the smooth real cubic surface surface are real, was previously studied by Yoshida \cite{YoshidaCubic} using the period map of \cite{ACTcubiccomplex}. The moduli space of real hyperelliptic curves of genus three has been studied by Chu \cite{Chu} using the period map of Deligne and Mostow \cite{DeligneMostow}. 

In this article we will focus mostly on smooth nonhyperelliptic curves of genus three. The canonical map of such a curve is an embedding onto a smooth plane quartic. For the moduli space of smooth plane quartic curves there is a period map due to Kondo \cite{Kondo1}. It maps the moduli space to a ball quotient of dimension six. We will study the question of reality for this period map. 

The classification of smooth real plane quartic curves is classical. The set of real points of such a curve consists of up to four ovals in the real projective plane. There are six possible configuration of the ovals. Each of them determines a connected component in the space of smooth real plane quartic curves. This is the projective space $P_{4,3}(\mathbb{R})=\mathbb{P}\Sym^4(\mathbb{R}^3)$ of dimension $14$ without the discriminant locus $\Delta(\mathbb{R})$, that represents singular quartics. Since the group $PGL_3(\mathbb{R})$ is connected, the moduli space
\[ \mathcal{Q}^\mathbb{R} =  PGL_3(\mathbb{R}) \backslash \left( P_{4,3}(\mathbb{R})-\Delta(\mathbb{R}) \right) \]
also consists of six components which we denote by $\mathcal{Q}_j^\mathbb{R}$ with $j=1,\ldots,6$. The correspondence between these components and the topological types of the set of real points of the curves is shown in Figure \ref{6concomp}.
\begin{table}
\centering
\begin{tabular}{m{1cm}m{1cm}m{1cm}m{1cm}m{1cm}m{1cm}}
\toprule
$\mathcal{Q}_1^\mathbb{R}$ & $\mathcal{Q}_2^\mathbb{R}$ & $\mathcal{Q}_3^\mathbb{R}$ &$\mathcal{Q}_4^\mathbb{R}$& $\mathcal{Q}_5^\mathbb{R}$& $\mathcal{Q}_6^\mathbb{R}$\\
\midrule
\begin{tikzpicture}\draw (45:10pt) circle [radius=5pt];\draw (135:10pt) circle [radius=5pt];\draw (-135:10pt) circle [radius=5pt];\draw (-45:10pt) circle [radius=5pt];\end{tikzpicture} &
\begin{tikzpicture}\draw (0:8pt) circle [radius=5pt];\draw (120:8pt) circle [radius=5pt];\draw (240:8pt) circle [radius=5pt];\end{tikzpicture} &
\begin{tikzpicture}\draw (0:7pt) circle [radius=5pt]; \draw (180:7pt) circle [radius=5pt];\end{tikzpicture} &
\begin{tikzpicture}\draw (0:0) circle [radius=5pt];\end{tikzpicture} &
\begin{tikzpicture}\draw (0,0) circle [radius=8pt];\draw (0,0) circle [radius=4pt];\end{tikzpicture}&$\emptyset $ 
 \\
\bottomrule
\end{tabular}
\caption{The topological types of representative curves $C(\mathbb{R})$ for the six components of $\mathcal{Q}_i^\mathbb{R} \subset \mathcal{Q}^\mathbb{R}$ for $i=1,\ldots,6$.}
\label{6concomp}
\end{table}

In this article we will prove that each of the components $\mathcal{Q}_j^\mathbb{R}$ is isomorphic to a divisor complement in an arithmetic real ball quotient. In order to formulate this more precisely we introduce some notation on Gaussian lattices. Let $\mathcal{G}=\mathbb{Z}[i]$ be the Gaussian integers and let $\Lambda_{1,6}$ be the Gaussian lattice $\mathcal{G}^7$ equipped with the Hermitian form $h(\cdot,\cdot)$ defined by the matrix
\begin{equation}\label{intromatrix} H = \begin{pmatrix} -2&1+i\\1-i&-2 \end{pmatrix}^{\oplus 3}\oplus (2). \end{equation}
We denote the group of unitary transformations of this lattice by $\Gamma=U(\Lambda)$. The lattice $\Lambda_{1,6}$ has hyperbolic signature $(1,6)$ and determines a complex ball of dimension six by the expression
\begin{equation} 
\mathbb{B}_6 = \mathbb{P} \{ z\in \Lambda_{1,6} \otimes_\mathcal{G} \mathbb{C} \ ; \ h(z,z)>0 \}. 
\end{equation}
A root is an element $r\in \Lambda_{1,6}$ such that $h(r,r)=-2$ and for every root $r$ we define its root mirror to be the hypersurface $H_r = \{ z\in \mathbb{B}_6 \ ; \ h(r,z)=0 \}$. We denote by $\mathbb{B}_6^\circ$ the complement in $\mathbb{B}_6$ of all root mirrors. Our main result is the following theorem.

\begin{thm}\label{introtheoremhyperbolic}There are six projective classes of antiunitary involutions $[\chi_j]$ with $j=1,\dots,6$ of the lattice $\Lambda_{1,6}$ up to conjugation by $P\Gamma$. Each of them determines a real ball $\mathbb{B}_6^{\chi_j}\subset \mathbb{B}_6$ and there are isomorphisms of real analytic orbifolds
\begin{equation}
\mathcal{Q}_j^\mathbb{R} \longrightarrow P\Gamma^{\chi_j} \backslash \left( \mathbb{B}_6^{\chi_j}  \right)^\circ.
\end{equation}
The group $P\Gamma^{\chi_j}$ is the stabilizer of the real ball $\mathbb{B}_6^{\chi_j}$ in $P\Gamma$. It is an arithmetic subgroup of $PO(\Lambda_{1,6}^{\chi_j})$ for each $j=1,\ldots,6$.
\end{thm}

In fact we obtain more information on the lattices $\Lambda_{1,6}^{\chi_j}$ and the groups $P\Gamma^{\chi_j}$ for $j=1,\ldots,6$. They are finite index subgroups of hyperbolic Coxeter groups of finite covolume and we determine the Coxeter diagrams for these latter groups using Vinberg's algorithm. 

For the group $P\Gamma^{\chi_1}$ that corresponds to the component $\mathcal{Q}_1^\mathbb{R}$ of maximal quartic curves we obtain a very explicit description: it is the semidirect product of a hyperbolic Coxeter group of finite covolume by its group of diagram automorphisms. The fundamental domain of this Coxeter group is a convex hyperbolic polytope $C_6$ whose Coxeter diagram is shown in Figure \ref{CoxDiag}. Its group of diagram automorphisms is the symmetric group $S_4$. The locus of fixed points in $C_6$ of this group is a hyperbolic line segment. It corresponds to a pencil of smooth real quartic curves that was previously studied by W.L. Edge \cite{Edge}. It consist of four ovals with an $S_4$-symmetry and we determine this family explicitly.

The walls of the polyhedron $C_6$ represent either singular quartics or hyperelliptic curves. The Coxeter diagram $C_5$ of the wall representing hyperelliptic curves is shown in Figure \ref{CoxDiag} on the right. It is the Coxeter diagram that corresponds to the connected component of the moduli space of real binary octics where all eight points are real. This component is described by Chu in \cite{Chu}. We complement this work by explicitly computing the Coxeter diagram of $C_5$. The automorphism group of this diagram is isomorphic to $D_8$ and there is a unique fixed point in $C_5$. It correspond to the isomorphism class of the binary octic where the zeroes are image of the eighth roots of unity under a Cayley transform $z \mapsto i(1-z)/(1+z)$.  

\begin{ack}
The results of this article are contained in the PhD thesis of the second author. This research was supported by NWO free competition grant number 613.000.909. The authors would like to thank Professor Allcock and Professor Kharlamov for useful comments. 
\end{ack}

\begin{figure}[H]
\centering
\begin{tabular}{m{8cm}m{3cm}}
\begin{tikzpicture}[scale=0.8]
    \node[c] (1) at ( 0,0) {};
    \node[c]             (2) at ( 6,0) {};    
    \node[c,] (3) at ( 8,0) {};
    \node[c]            (4) at ( 5.5,-1.25) {};
    \node[c] (5) at ( 4,-2) {};
    \node[c]             (6) at ( 2,-1) {};    
    \node[c] (7) at ( 4,4) {};
    \node[c]             (8) at ( 4,2) {};    
    \node[c]             (9) at ( 6.25,1.75) {};
    \node[c]             (10) at (1.5,1.5) {};
    
    \node[c] (x1) at (4.5,1.5) {};
    \node[c] (x2) at (4.5,0.5) {};
    \node[c] (x3) at (3,0.5) {};
    
    \draw [-,double distance=2pt] (1) -- (2) -- (3);
    
    \draw [-,ultra thick] (x1) -- (x2) -- (x3) -- (x1);
    \draw [-,ultra thick] (2) -- (x1) -- (8);
    \draw [-,ultra thick] (4)  --(x3) -- (10);
    \draw [-,draw=white,line width=4pt] (6) -- (x2) -- (9);
    
    \draw [-,ultra thick] (6) -- (x2) -- (9);

    \draw [-,draw=white,line width=5pt] (5) -- (8) -- (7);
    \draw [-,double distance=2pt] (5) -- (8) -- (7)  ;
    
    \draw [-,double distance=2pt] (5) -- (4) -- (3) ;
    \draw [-,double distance=2pt] (5) -- (6) -- (1);
    \draw [-,double distance=2pt] (1) -- (10) -- (7);
    \draw [-,double distance=2pt] (3) -- (9) -- (7);   
\end{tikzpicture}%
&
\begin{tikzpicture}
    \node[c] (1) at ( 0,-1) {};
    \node[c] (2) at ( 1,-1) {};    
    \node[c] (3) at ( 2,-1) {};
    \node[c] (4) at ( 2,0) {};
    \node[c] (5) at ( 2,1) {};
    \node[c] (6) at ( 1,1) {};
    \node[c] (7) at ( 0,1) {};
    \node[c] (8) at ( 0,0) {}; 
    \draw [-,double distance = 2pt] (1) -- (2) -- (3) -- (4) -- (5) -- (6) -- (7) -- (8) -- (1);
\end{tikzpicture}
\end{tabular}
\caption{The Coxeter diagram of the hyperbolic Coxeter polytope $C_6$ (left) and the wall that corresponds to $C_5$ (right). The nodes represent the walls and a double edge connecting two nodes means their walls meet at an angle of $\pi/4$, a thick edge means they are parallel and no edge means they are orthogonal.}
\label{CoxDiag}
\end{figure}
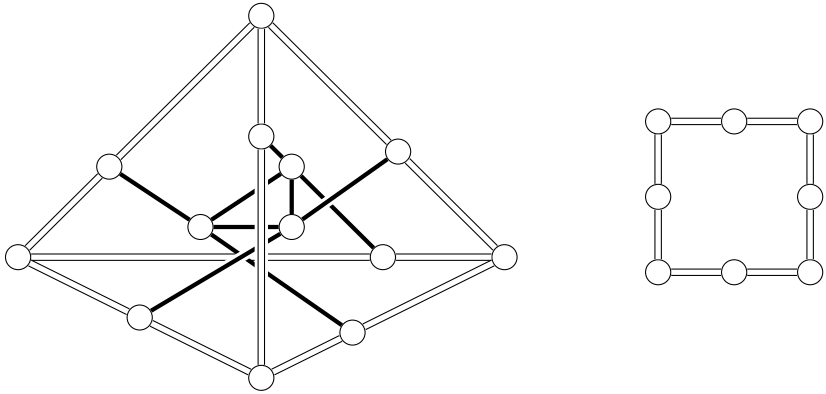

\section{Lattices}\label{Lattices}

A lattice is a pair $(L,(\cdot,\cdot))$ with $L$ a free $\mathbb{Z}$-module of finite rank $r$ and $(\cdot,\cdot)$ a nondegenerate, symmetric bilinear form on $L$ taking values in $\mathbb{Z}$. This bilinear form extends naturally to a bilinear form $(\cdot,\cdot)_\mathbb{Q}$ on the rational vector space $L\otimes_\mathbb{Z} \mathbb{Q}$ and its signature $(r_+,r_-)$ is called the signature of $L$. The dual of $L$ is the group $L^\vee=\Hom(L,\mathbb{Z})$ and the lattice $L$ is naturally embedded in $L^\vee$ by the assignment $x\mapsto (x,\cdot)$.  The group $L^\vee$ is naturally embedded in the vector space $L\otimes_\mathbb{Z} \mathbb{Q}$ by the identification
\[ L^\vee = \{ x\in L\otimes_\mathbb{Z} \mathbb{Q} \ ; \ (x,y)_\mathbb{Q} \in \mathbb{Z} \text{ for all } y\in L \}.\]
Note that the induced bilinear form on $L^\vee$ need not be integer valued, but by abuse of language we still call $L^\vee$ a lattice. An isomorphism between lattices $L_1$ and $L_2$ is a group isomorphism $\phi:L_1\rightarrow L_2$ that preserves the bilinear forms of $L_1$ and $L_2$. If $\{e_1,\ldots,e_r\}\subset L$ is a basis for $L$ then the matrix 
\[B=\begin{pmatrix}(e_1,e_1)&\cdots& (e_1,e_r)\\
\vdots &\ddots & \vdots \\
(e_r,e_1)&\cdots& (e_r,e_r)\end{pmatrix}\] 
is called the Gram matrix. Its determinant $d(L)$ is an invariant called the discriminant of the lattice. A lattice is called unimodular if $d(L)=\pm 1$ or equivalently if $L^\vee=L$. A lattice $L$ is called even if $(x,x)\in 2\mathbb{Z}$ for all $x\in L$, otherwise it is called odd.  We denote the automorphism group of a lattice $L$ by $O(L)$. An important class of automorphisms of a lattice $L$ of signature $(r_+,r_-)$ with $r_+ \leq 1$ is given by its reflections. For $r \in L$ primitive (that is $q\cdot r \in L$ for $q\in \mathbb{Q}$ only if $q\in \mathbb{Z}$) and of negative norm $(r,r)$ we define the reflection $s_r$ in $r$ by the formula
\begin{equation}s_r(x) = x-2\frac{(r,x)}{(r,r)}r.\end{equation}
This reflection is an automorphism of the lattice $L$ if and only if $2(r,x)\in (r,r)\mathbb{Z}$ for all $x\in L$. In that case we call the negative norm vector $r$ a root in $L$. Since conjugation by an element of $O(L)$ of a reflection is again a reflection, the reflections in roots generate a normal subgroup $W(L)\triangleleft O(L)$. 

Let $L$ be an even lattice. The quotient $A_L=L^\vee / L$ is called the discriminant group of $L$. It is a finite abelian group of order $d(L)$. We denote the minimal number of generators of $A_L$ by $l(A_L)$. If $A_L\cong (\mathbb{Z}/2\mathbb{Z})^a$ for some $a\in \mathbb{N}$ then $L$ is called $2$-elementary.    

\begin{prop}[Nikulin \cite{NikulinSym}, Thm. 3.6.2]\label{2invariants}An indefinite, even $2$-elementary lattice with $r_+>0$ and $r_->0$ is determined up to isomorphism by the invariants $(r_+,r_-,a,\delta)$. The invariant $\delta$ is defined by
\[ \delta = 
\begin{cases}0 & \text{ if } (x,x)_\mathbb{Q} \in \mathbb{Z} \text { for all } x\in L^\vee \\
1 & \text{ else }\end{cases}
\]
\end{prop}

The discriminant quadratic form $q_L$ on $A_L$ takes values in $\mathbb{Q}/2\mathbb{Z}$ and is defined by the expression
\begin{align*}
q_L(x+L)  &\equiv (x,x)_\mathbb{Q} \bmod{2\mathbb{Z}} \quad \text{for } x\in L^\vee. \\
\end{align*} 
The group of automorphisms of $A_L$ that preserve the discriminant quadratic form $q_L$ is denoted by $O(A_L)$ and there is a natural homomorphism: $O(L)\rightarrow O(A_L)$. If $\phi_L \in O(L)$ then we denote by $q(\phi_L)\in O(A_L)$ the induced automorphism of $A_L$. 

\begin{thm}[Nikulin, \cite{NikulinSym}, Thm. 3.6.3]\label{Nik363}Let $L$ be an even, indefinite $2$-elementary lattice. Then the natural homomorphism $O(L) \rightarrow O(A_L)$ is surjective.
\end{thm}

\begin{prop}[Nikulin, \cite{NikulinSym}, Prop. 1.6.1]\label{gluing}Let $L$ be an even unimodular lattice and $M$ a primitive sublattice of $L$ with orthogonal complement $M^\perp=N$. There is a natural isomorphism $\gamma:A_M\rightarrow A_N$ for which $q_N\circ \gamma = -q_M$. Let $\phi_M \in O(M)$ and $\phi_N\in O(N)$. The automorphism $(\phi_M , \phi_N)$ of $M\oplus N$ extends to $L$ if and only if $q(\phi_N) \circ \gamma =\gamma \circ q(\phi_M)$. 
\end{prop}

\begin{thm}[Nikulin, \cite{NikulinSym}, Thm. 1.14.4]\label{primembed}Let $M$ be an even lattice of signature $(s_+,s_-)$ and let $L$ be an even unimodular lattice of signature $(r_+,r_-)$. There is a unique primitive embedding of $M$ into $L$ provided the following hold:
\begin{enumerate}
\item $s_+<r_+$
\item $s_-<r_-$
\item $l(A_M) \leq \rank(L) - \rank(M)-2$
\end{enumerate}
\end{thm}


We denote by $L(n)$ the lattice $L$ where the bilinear form is scaled by a factor $n\in \mathbb{Z}$ and we write $U$ for the even unimodular hyperbolic lattice of rank $2$ with Gram matrix $\left( \begin{smallmatrix}0&1\\1&0\end{smallmatrix} \right)$. Furthermore we denote by $A_i,D_j,E_k$ with $i,j\in \mathbb{N}$, $j\geq 4$ and $k=6,7,8$ the lattices associated to the negative definite Cartan matrices of this type. For example
\[
A_2 = \begin{pmatrix}-2&1\\1&-2\end{pmatrix} \ ,  \ A_1\oplus A_1(2) = \begin{pmatrix}-2&0\\0&-4\end{pmatrix} \ , \ D_4 = \begin{pmatrix}-2&1&0&0\\1&-2&1&1\\0&1&-2&0\\0&1&0&-2\end{pmatrix}.
\]

Determining if two lattices are isomorphic can be challenging. In the following lemma we describe some isomorphic lattices that we will encounter frequently when studying Gaussian lattices.
\begin{lem}\label{simplifylattices}
There are isomorphisms of hyperbolic lattices
\begin{align}
(4) \oplus A_1 &\cong (2) \oplus A_1(2) \\
U(2) \oplus A_1  &\cong (2) \oplus A_1^2 \\
(2) \oplus A_1(2) \oplus D_4(2) &\cong (2) \oplus A_1^2 \oplus A_1(2)^3
\end{align}
\end{lem}

\begin{proof}
For the first isomorphism we explicitly determine a base change:
\[ \begin{pmatrix}1&-1\\-1&2\end{pmatrix}^t \begin{pmatrix}4&0\\0&-2\end{pmatrix} \begin{pmatrix}1&-1\\-1&2\end{pmatrix} =\begin{pmatrix}2&0\\0&-4\end{pmatrix}.\]
For the second isomorphism we calculate the invariants $(r_+,r_-,a,\delta)$ of Proposition \ref{2invariants}. They are easily seen to be $(1,2,3,1)$ for both lattices so that the lattices are isomorphic. The third isomorphism is the least obvious. We also determine an explicit base change:
\[ B^t \left( (2) \oplus A_1^2 \oplus A_1(2)^3 \right)B = (2) \oplus A_1(2) \oplus D_4(2) \]
where $B$ is the unimodular matrix:
\[ B = \begin{pmatrix}3&2&1&0&1&1\\-1&0&-1&1&-1&-1\\-1&0&0&-1&0&0\\-1&-1&0&0&0&-1\\-1&-1&0&0&-1&0\\-1&-1&-1&0&0&0\end{pmatrix}.\]
\end{proof}


\section{Hyperbolic reflection groups}\label{Vinberg}

Most of the results of this section can be found in \cite{VinbergAlg}. Let $L$ be a hyperbolic lattice of hyperbolic signature $(1,n)$. We can associate to $L$ the space $V=L\otimes_\mathbb{Z} \mathbb{R}$ with isometry group $O(V)\cong O(1,n)$. A model for real hyperbolic $n$-space $\mathbb{H}_n$ is given by one of the sheets of the two sheeted hyperboloid $\{ x\in V \ ; \ (x,x)=1 \}$ in $V$. Its isometry group is the subgroup $O(V)^+<O(V)$ of index two of isometries that preserves this sheet. Another model for $\mathbb{H}_n$ which we will use most of the time is the ball defined by
\[ \mathbb{B}_n = \mathbb{P}\{ x\in L\otimes_\mathbb{Z} \mathbb{R} \ ; \ (x,x)>0 \} \] 
whose isometry group is naturally identified with the group $O(\mathbb{B})\cong PO(1,n)$.
The group $O(L)^+=O(L)\cap O(V)^+$ is a discrete subgroup of $O(V)^+$ and it has finite covolume by a theorem of Siegel \cite{Siegel}. Let  $W(L)<O(L)^+$ be the normal subgroup generated by the reflections in roots of negative norm of $L$. We can write the group $O(L)^+$ as
\[O(L)^+ = W(L) \rtimes S(C)\]
where $C\subset \mathbb{B}_n$ is a fundamental chamber of $W(L)$ and $S(C)$ is the subgroup of $O(L)^+$ that maps $C$ to itself. The lattice $L$ is called reflective if $W(L)$ has finite index in $O(L)^+$. In this case $C$ is a hyperbolic polytope of finite volume which we assume from now on. We say that $\{r_i\}_{i\in I}$ with $I=\{1,\dots,k\}$ is a set of simple roots for $C$ if all pairwise inner products are nonnegative and $C$ is the polyhedron bounded by the mirrors $H_{r_i}$ so that
\begin{equation}\label{polyhedronC}
C = \{ x \in L\otimes_{\mathbb{Z}} \mathbb{R} \ ; \ (x,x)>0 \ ,\ (r_i,x) \geq 0 \text{ for } i=1,\ldots,k \} / \mathbb{R}_+.
\end{equation}
The root mirrors meet at dihedral angles $\pi/m_{ij}$ with $m_{ij} = 2,3,4,6$ or they are disjoint in $\mathbb{B}_n$. In this last case we say that two root mirrors $H_{r_i}$ and $H_{r_j}$ are parallel if they meet at infinity so that $m_{ij}=\infty$, or ultraparallel if they do not meet even at infinity. The matrix $G$ with entries $(r_i,r_j)_{i,j\in I}$ is called the Gram matrix of $C$ and in case two mirrors are not ultraparallel the $m_{ij}$ can be calculated from $G$ by the relation
\[ (r_i,r_j)^2 = (r_i,r_i)(r_j,r_j)\cos^2 \left( \frac{\pi}{m_{ij}} \right). \]
The polytope $C$ is described most conveniently by its Coxeter diagram $D_I$. This is a graph with $k$ nodes labeled by simple roots $\{r_i\}_{i\in I}$. Nodes $i$ and $j$ are connected by $4\cos^2(\pi/m_{ij})$ edges in case $m_{ij} < \infty$. If $m_{ij}=\infty$ we connect the vertices by a thick edge. In addition we connect two nodes by a dashed edge if their corresponding mirrors are ultraparallel. In the examples that come from Gaussian lattices we will only encounter roots of norm $-2,-4$ and $-8$ so we also subdivide the corresponding nodes into $0,2$ and $4$ parts respectively. These conventions are illustrated in Figure \ref{conv}.
\begin{figure}
\centering
\begin{tabular}{llll}
\begin{tikzpicture}[baseline=-2pt]\node[c] (1) at ( 0,0) {};\node[c] (2) at (1,0) {};\end{tikzpicture}& orthogonal intersection & \begin{tikzpicture}[baseline=-2pt]\node[c] (1) at (0,0) {};\end{tikzpicture} & $(r,r)=-2$ \\
\begin{tikzpicture}[baseline=-2pt]\node[c] (1) at ( 0,0) {};\node[c] (2) at (1,0) {};\draw[-] (1) -- (2); \end{tikzpicture} & interior angle $\pi/3$ & \begin{tikzpicture}[baseline=-2pt]\node[c,vert] (1) at (0,0) {};\end{tikzpicture} & $(r,r)=-4$\\
\begin{tikzpicture}[baseline=-2pt]\node[c] (1) at ( 0,0) {};\node[c] (2) at (1,0) {};\draw[-,double distance =2pt] (1) -- (2);\end{tikzpicture} & interior angle $\pi/4$  & \begin{tikzpicture}[baseline=-2pt]\node[c,oplus] (1) at (0,0) {};\end{tikzpicture} & $(r,r)=-8$\\
\begin{tikzpicture}[baseline=-2pt]\node[c] (1) at ( 0,0) {};\node[c] (2) at (1,0) {};\draw[-,ultra thick] (1) -- (2);\end{tikzpicture} & parallel & & \\
\begin{tikzpicture}[baseline=-2pt]\node[c] (1) at ( 0,0) {};\node[c] (2) at (1,0) {};\draw[dashed] (1) -- (2);\end{tikzpicture} & ultraparallel & &\\
\end{tabular}
\caption{Conventions for Coxeter graphs}
\label{conv}
\end{figure}
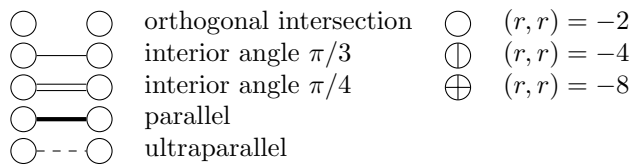

A Coxeter subdiagram $D_J\subset D_I$ with $J\subset I$ is called elliptic if the corresponding Gram matrix is negative definite of rank $\vert J \vert$ and parabolic if it is negative semidefinite of rank $\vert J\vert -\#$components of $D_J$. An elliptic subdiagram is a disjoint union of finite Coxeter diagrams and a parabolic subdiagram the disjoint union of affine Coxeter diagrams. The elliptic subdiagrams of $D$ of rank $r$ correspond to the $(n-r)$-faces of the polyhedron $C\in \mathbb{B}_n$. A parabolic subdiagram of rank $n-1$ corresponds to a cusp of $C$. By the type of a face or cusp of $C$ we mean the type of the corresponding Coxeter subdiagram. 

\subsection{Vinberg's algorithm}

Suppose we are given a hyperbolic lattice $L$  of signature $(1,n)$. Vinberg \cite{VinbergAlg} describes an algorithm to determine a set of simple roots of $W(L)$. If the algorithm terminates these simple roots determine a hyperbolic polyhedron $C\subset \mathbb{B}_n$ of finite volume which is a fundamental chamber for the reflection subgroup $W(L)$. We start by choosing a controlling vector $p\in L$ such that $(p,p)>0$. This implies that $[p]\in \mathbb{B}_n$. The idea is to determine a sequence of roots $r_1,r_2,\ldots$ so that the hyperbolic distance of $p$ to the mirrors $H_{r_i}$ is increasing. Since the hyperbolic distance $d(p,H_{r_i})$ is given by
\begin{equation}\label{hypdistance}
\sinh^2 d(p,H_{r_i}) = \frac{ -(r_i,p)^2}{(r_i,r_i)\cdot (p,p)}
\end{equation}
the height $h(r_i)$ of a root defined by $h(r_i)=-2(r_i,p)^2/(r_i,r_i)$ is a measure for this distance. First we determine the roots of height $0$. They form a finite root system $R$ and we choose a set of simple roots $r_1,\ldots,r_i$ to be our first batch of roots. For the inductive step in the algorithm we consider all roots of height $h$ and assume that all roots of smaller height have been enumerated. A root of height $h$ is accepted if and only if it has nonnegative inner product with all previous roots of the sequence. The algorithm terminates if the accute angled polyhedron spanned by the mirrors $H_{r_1},\ldots$ has finite volume. This can be checked using the following criterion also due to Vinberg.

\begin{prop}A Coxeter polyhedron $C\subset \mathbb{B}_n$ has finite volume if and only if every elliptic subdiagram of rank $n-1$ can be extended in exactly two ways to an elliptic subdiagram of rank $n$ or to a parabolic subdiagram of rank $n-1$. Furthermore there should be at least $1$ elliptic subdiagram of rank $n-1$.
\end{prop}

Since an elliptic subdiagram of rank $n-1$ corresponds to an edge of the polyhedron $C$ the geometrical content of this criterion is that every edge connects either two actual vertices, two cusps or a vertex and a cusp. The following example is due to Vinberg, see \cite{Vinbergquad} \S 4.

\begin{ex}\label{VinCoxex}Consider the hyperbolic lattice $\mathbb{Z}_{1,n}(2)=(2)\oplus A_1^n$ with its standard orthogonal basis $\{e_0,\ldots,e_n\}$ where $2\leq n \leq 9$. The possible root norms are $-2$ and $-4$. We take as controlling vector $p=e_0$ with $(p,p)=2$. The height $0$ root system is of type $B_n$ and a basis of simple roots is given by
\begin{align*}
r_1=e_1-e_2 \ , \ \ldots, r_n = e_{n-1}-e_n \ , \ r_n= e_n.
\end{align*}
The next root accepted by Vinberg's algorithm is the root $r_{n+1}=e_0-e_1-e_2$ of height $2$ for $n=2$ and the root $r_{n+1}=e_0-e_1-e_2-e_3$ of height $1$ for $3\leq n \leq 9$. This root indeed satisfies $(r_{n+1},r_i)>0$ for $i=1,\ldots,n$. The resulting Coxeter polyhedron is a simplex and has finite volume so the algorithm terminates. In all the cases there is a single cusp of type $\widetilde{A}_1$ for $n=2$ and of type $\widetilde{B}_{n-1}$ for $n=3,\ldots,8$, except when $n=9$ in which case there are 2 cusps of type $\widetilde{B}_{8}$ and $\widetilde{E}_8$. The Coxeter diagrams are shown in Figure \ref{VinCox}.
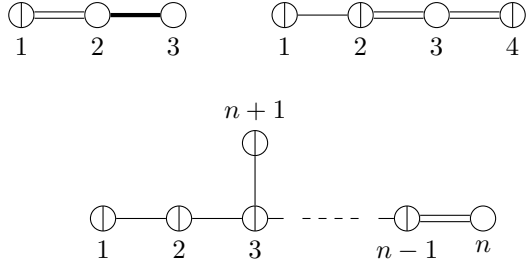
\begin{figure}\centering
\subfigure{
\begin{tikzpicture}
\node[c,vert,label = below:$1$] (1) at (0,0) {};
\node[c,label = below:$2$] (2) at (1,0) {};
\node[c,label = below:$3$] (3) at (2,0) {};
\draw[-,double distance = 2pt] (1) -- (2);
\draw[-,ultra thick] (2) -- (3);
\end{tikzpicture}
}
\qquad
\subfigure{
\begin{tikzpicture}
\node[c,vert,label = below:$1$] (1) at (0,0) {};
\node[c,vert,label = below:$2$] (2) at (1,0) {};
\node[c,label = below:$3$] (3) at (2,0) {};
\node[c,vert,label = below:$4$] (4) at (3,0) {};
\draw[-] (1) -- (2);
\draw[-,double distance = 2pt] (2) -- (3) -- (4);
\end{tikzpicture}
}

\qquad
\subfigure{
\centering
\begin{tikzpicture}
\node[c,vert,label = below:$1$] (1) at (0,0) {};
\node[c,vert,label = below:$2$] (2) at (1,0) {};
\node[c,vert,label = below:$3$] (3) at (2,0) {};
\node (x1) at (2.5,0) {};
\node (x2) at (3.5,0) {};
\node[c,vert,label=below:$n-1$] (4) at (4,0) {};
\node[c,label = below:$n$] (5) at (5,0) {};
\node[c,vert,label=above:$n+1$] (6) at (2,1) {};
\draw[-] (1) -- (2) -- (3) -- (6);
\draw[-] (3) -- (x1);
\draw[dashed] (x1) -- (x2);
\draw[-] (x2) -- (4);
\draw[-,double distance = 2pt] (4) -- (5);
\end{tikzpicture}
}
\caption{The Coxeter diagrams of the groups $O(\mathbb{Z}_{1,n}(2))^+$ for $n=2,\ldots,9$.}
\label{VinCox}
\end{figure}    
\end{ex}

\section{Gaussian lattices}

\label{Gaussian}

This section is in a sense the technical heart of this article. We study Gaussian lattices of hyperbolic signature and show how these give rise to arithmetic complex ball quotients. Antiunitary involutions of the Gaussian lattice then correspond to real forms of these ball quotients. The main examples are the two Gaussian lattices $\Lambda_{1,5}$ and $\Lambda_{1,6}$ whose ball quotients correspond to the moduli spaces of smooth binary octics and smooth quartic curves. An excellent reference on the topic of Gaussian lattices is \cite{Allcock}. It also contains many examples of lattices over the Eisenstein and Hurwitz integers.

A Gaussian lattice is a pair $(\Lambda,\rho)$ with $\Lambda$ a lattice and $\rho \in O(\Lambda)$ an automorphism of order four such that the powers $\rho,\rho^2$ and $\rho^3$ act without nonzero fixed points. Such a lattice $\Lambda$ can be considered as a module over the ring of Gaussian integers $\mathcal{G}=\mathbb{Z}[i]$ by assigning $(a+ib)x=ax+b\rho(x)$ for all $x\in \Lambda$ and $a,b\in \mathbb{Z}$. The expression 
\[h(x,y)=(x,y) + i(\rho(x),y)\]
defines a $\mathcal{G}$-valued nondegenerate Hermitian form on $\Lambda$ which is linear in its second argument and antilinear in its first argument. 
Conversely suppose that $\Lambda$ is a free $\mathcal{G}$-module of finite rank equipped with a $\mathcal{G}$-valued Hermitian form $h(\cdot,\cdot)$. We define a symmetric bilinear form on the underlying $\mathbb{Z}$-lattice of $\Lambda$ by taking the real part of the Hermitian form: $(x,y)=\Real h(x,y)$. Multiplication by $i$ defines an automorphism $\rho$ of order $4$ so the pair $(\Lambda,\rho)$ is a Gaussian lattice. It is easily checked that these two constructions are inverse to each other. Another way of defining a Gaussian lattice is by prescribing a Hermitian Gaussian matrix. Such a matrix $H$ satisfies $\overline{H}^t=H$ and defines a Hermitian form on $\mathcal{G}^n$ by the formula $h(x,y)=\bar{x}^tHy$.
The dual of a Gaussian lattice $\Lambda$ is the lattice $\Lambda^\vee = \Hom(\Lambda,\mathcal{G})$. It is naturally embedded in the vector space $\Lambda \otimes_\mathcal{G} \mathbb{Q}(i)$ by the identification
\[ \Lambda^\vee = \left\{ x \in \Lambda \otimes_\mathcal{G} \mathbb{Q}(i) \ ; \ h(x,y)\in \mathcal {G} \ \text{for all} \ y\in \Lambda \right\} .\] 

From now on we only consider nondegenerate Gaussian lattices that satisfy the condition $h(x,y)\in (1+i)\mathcal{G}$ for al $x,y\in \Lambda$. This is equivalent to $\Lambda \subset (1+i)\Lambda^\vee$ and implies that the underlying $\mathbb{Z}$-lattice of $\Lambda$ is even. 

\begin{lem}\label{rhocommutes}The group $U(\Lambda)$ of unitary transformations of a Gaussian lattice $\Lambda$ is equal to the group
\[ \Gamma = \{ \gamma \in O(\Lambda) \ ; \ \gamma \circ \rho = \rho \circ \gamma \} \] 
of orthogonal transformations of the underlying $\mathbb{Z}$-lattice of $\Lambda$ that commute with $\rho$.
\end{lem}
\begin{proof}
If $\gamma \in U(\Lambda)$ then by definition $h(\gamma x,\gamma y) = h(x,y)$ for all $x,y\in \Lambda$. Using the definition of the Hermitian form $h$ this is equivalent to
\begin{align*}
(\gamma x,\gamma y) + i (\rho \gamma x, \gamma y)
&= (x,y) + i(\rho x,y).\\
\end{align*}
By considering the real part of this equality we see that $\gamma \in O(\Lambda)$. Combining this with the equality of the imaginary parts of the equation we obtain $(\rho \gamma x,\gamma y)=(\gamma \rho x,\gamma y)$ for all $x,y \in \Lambda$. This is equivalent to: $\rho \circ \gamma=\gamma \circ \rho$. This proves the inclusion $U(\Lambda)\subset \Gamma$. For the other inclusion we can reverse the argument. 
\end{proof}

A root $r \in \Lambda$ is a primitive element of norm $-2$. For every root $r$ we define a complex reflection $t_{r}$ of order $4$ (a tetraflection) by
\begin{equation}
t_{r}(x) = x - (1-i)\frac{h(r,x)}{h(r,r)} r ,
\end{equation}
which is an element of $U(\Lambda)$ because $\Lambda \subset (1+i)\Lambda^\vee$. It is a unitary transformation of $\Lambda$ that maps $r \mapsto ir$ and fixes pointwise the mirror $H_r=\{x\in \Lambda \ ; \ h(r,x)=0 \}$. The tetraflection $t_{r}$ and the mirror $H_r$ only depend on the orbit of $r$ under the group of units $\mathcal{G}^\ast=\{ 1,i,-1,-i\}$ of $\mathcal{G}$. We call such an orbit a projective root and denote it by $[r]$. If the group generated by tretraflections in the roots has finite index in $\Gamma = U(\Lambda)$ we say that the lattice $\Lambda$ is tetraflective.

\begin{ex}[The Gaussian lattice $\Lambda_2$]
The lattice $D_4$ is given by
\[ D_4 = \left\{ x \in \mathbb{Z}^4 \ ; \ \sum x_i \equiv 0 \bmod{2} \right\} \]
with the symmetric bilinear form induced by the standard form of $\mathbb{Z}^4$ scaled by a factor $-1$ so that $(x,y)=-\sum x_iy_i$. We choose a basis for this lattice given by the roots $\{\beta_i \}$ with the Gram matrix $B$ shown below.
\[
\begin{array}{lcr}
\beta_1=e_3-e_1 & \multirow{3}{*}{$\  \ B=\begin{pmatrix}-2&0&1&-1\\0&-2&1&1\\1&1&-2&0\\-1&1&0&-2\end{pmatrix}$} &  \multirow{3}{*}{$\  \ \rho =\begin{pmatrix}0&-1&0&0\\1&0&0&0\\0&0&0&-1\\0&0&1&0\end{pmatrix}$}\\
\beta_2=-e_1-e_3 \\
\beta_3=e_1-e_2 \\
\beta_4=e_3-e_4
\end{array}
\]
The matrix $\rho$ defines an automorphism of order $4$ without fixed points which turns the lattice $D_4$ into a Gaussian lattice which we will call $\Lambda_2$. A basis for $\Lambda_2$ is given by the roots $\{\beta_1,\beta_3\}$ and the Gram matrix $H$ with respect to this basis is given by
\[ H = \begin{pmatrix}-2&1+i\\1-i&-2\end{pmatrix}. \]
A small calculation shows that there are $6$ projective roots which are the $\mathcal{G}^\ast$-orbits of the roots
\[ \left\{ \begin{pmatrix}1\\0\end{pmatrix} , \begin{pmatrix}0\\1\end{pmatrix} , \begin{pmatrix}1\\1\end{pmatrix} , \begin{pmatrix}1\\-i\end{pmatrix}, \begin{pmatrix}1\\1-i\end{pmatrix} , \begin{pmatrix}1+i\\1\end{pmatrix} \right\} \]
The group generated by the tetraflections in these roots is the complex reflection group $G_8$ of order $96$ in the Shephard-Todd classification \cite{Shephard}. A basis for the dual lattice $\Lambda_2^\vee$ is given by $\left\{\frac{1}{1+i}\beta_1,\frac{1}{1+i}\beta_3 \right\}$ so that $(1+i)\Lambda_2^\vee = \Lambda_2$. \end{ex}

\begin{ex}Consider the Gaussian lattice $\Lambda_{1,1}$ with basis $\{e_1,e_2\}$ and Hermitian form defined by the matrix:
\[H=\begin{pmatrix} 0&1+i\\ 1-i&0 \end{pmatrix}.\]
It is easy to verify that $(1+i)\Lambda_{1,1}^\vee=\Lambda_{1,1}$. A basis $\{\beta_1,\ldots,\beta_4\}$ for the underlying $\mathbb{Z}$-lattice and its Gram matrix $B$ are shown below.
\[
\begin{array}{lr}
\beta_1 = e_2 & \multirow{3}{*}{$B=\begin{pmatrix}0&1&0&0\\1&0&0&0\\0&0&0&2\\0&0&2&0\end{pmatrix}$} \\
\beta_2 = ie_1  \\
\beta_3 = e_1-ie_1  \\
\beta_4 = e_2-ie_2 
\end{array}
\]
We conclude that the underlying $\mathbb{Z}$-lattice is isomorphic to $U\oplus U(2)$. 
\end{ex}
Using these two examples of Gaussian lattices we can construct many more by forming direct sums. We are especially interested in the Gaussian lattices of hyperbolic signature since these occur in the study of certain moduli problems. For example the Gaussian lattice $ \Lambda_{1,1} \oplus \Lambda_2 \oplus \Lambda_2 $ plays an important role in the study of the moduli space $\mathcal{M}_{0,8}$ of $8$ points on the projective line. Yoshida and Matsumoto \cite{MatsumotoYoshida} prove that the unitary group of this lattice is generated by $7$ tetraflections so that it is in particular tetraflective. This also follows from the work of Deligne and Mostow \cite{DeligneMostow}.

\subsection{Antiunitary involutions of Gaussian lattices}
Let $\Lambda$ be a Gaussian lattice of rank $n$ and signature $(n_+,n_-)$. An antiunitary involution of $\Lambda$ is an involution $\chi$ of the underlying $\mathbb{Z}$-lattice that satisfies
\[ h\left( \chi(x),\chi(y) \right) = \overline{h\left( x,y \right)}. \] 
Equivalently it is an involution that anticommutes with $\rho$ so that: $\chi \circ \rho = - \rho \circ \chi$. An antiunitary involution $\chi$ naturally extends to the $\mathbb{Q}(i)$-vectorspace $\Lambda_{\mathbb{Q}}=\Lambda \otimes_\mathcal{G} \mathbb{Q}(i)$ which can be regarded as a $\mathbb{Q}$-vectorspace of dimension $2n$ and signature $(2n_+,2n_-)$. The fixed point subspace $\Lambda_\mathbb{Q}^\chi$ is a $\mathbb{Q}$-vectorspace of dimension $n$ and signature $(n_+,n_-)$. Consider the fixed point lattice $\Lambda^\chi = \Lambda \cap \Lambda_\mathbb{Q}^\chi$. The Hermitian form restricted to $\Lambda^\chi$ takes on real values in $(1+i)\mathcal{G}$ and therefore has in in fact even values. This implies that $\Lambda^\chi(\frac{1}{2})$ in an integral lattice. 
\begin{prop}\label{antiunit}
Let $\Lambda$ be the Gaussian lattice defined by a Hermitian matrix $H$, so that in particular $\Lambda \cong \mathcal{G}^n$. Every antiunitary involution of $\Lambda$ is of the form $\chi= M \circ \conj$ where $\conj$ is standard complex conjugation on $\Lambda \cong \mathcal{G}^n$. The matrix $M$ has coefficients in $\mathcal{G}$ and satisfies $\overline{M}M=I$ and: $\overline{M}^tHM=\overline{H}$.
\end{prop} 

\begin{proof}
Suppose that $\chi$ is a antiunitary involution of $\Lambda$. Since every antiunitary involution on the vector space $\Lambda \otimes_\mathcal{G} \mathbb{Q}(i)$ is conjugate to standard complex conjugation there is a matrix $N$ such that $N\circ \chi \circ N^{-1} =\conj$. We can rewrite this as $\chi = M \circ \conj$ where $M=N^{-1}\overline{N}$ and $M$ has coefficients in $\mathcal{G}$. It is clear that $\overline{M}M=I$. Finally we can rewrite the equality $h(\chi(x),\chi(y)) = \overline{h(x,y)}$ as
\[ x^t\overline{M}^tHM\overline{y} = x^t\overline{H}y. \]
This holds for all $x,y\in \Lambda$ so that the last equality of the proposition follows.
\end{proof}

Let $\chi$ be an antiunitary involution of the lattice $\Lambda$ and let $[\chi]$ be its projective equivalence class. The elements of $[\chi]$ are the involutions $i^k\chi$ with $k=0,1,2,3$. By conjugation with the scalar $i$ we see that the two involutions $\{\chi,-\chi\}$ and also $\{i\chi,-i\chi\}$ are conjugate in $\Gamma$. The antiunitary involutions $\chi$ and $i\chi$ need not be $\Gamma$-conjugate, so in particular their fixed point lattices need not be isomorphic. This can already be seen in the simplest case of antiunitary involutions on $\mathcal{G}$. The fixed points lattice of the antiunitary involution $\conj$ and $i\circ \conj$ are $\mathbb{Z}$ and $(1+i)\mathbb{Z}$ respectively. 

We now present some computational lemma's on antiunitary involutions of Gaussian lattices of small rank. These will be very useful later on and will be referenced to throughout this text. 

\begin{lem}\label{fixedblocks}Let $\psi_1,\psi_2,\psi_2^\prime$ and $\psi_4$ be the transformations obtained by composing the following matrices with complex conjugation:
\[ M_1 = (1) \ , \ M_2=\begin{pmatrix}i&0\\0&1\end{pmatrix} \ , \ M_2^\prime=\begin{pmatrix}0&1\\1&0\end{pmatrix} \ , \ M_4=\begin{pmatrix}0&0&i&0\\0&0&0&1\\i&0&0&0\\0&1&0&0\end{pmatrix}. \]
They define antiunitary involutions $\chi$ on certain Gaussian lattices $\Lambda$ shown in Table \ref{fixedblockstable}. The fixed point lattices $\Lambda^\chi$ are also computed along with a matrix $B^\chi$ such that the columns of this matrix form a $\mathbb{Z}$-basis for the fixed point lattice $\Lambda^\chi$. 
\begin{table}
\centering
\[
\begin{array}{llll}
\toprule
\chi & \Lambda & \Lambda^\chi & B^\chi \\
\midrule
\psi_1 & \mathcal{G} & A_1 & (1) \\
i\psi_1 & \mathcal{G} & A_1(2) & (1+i) \\
\psi_2 & \Lambda_2 & A_1^2 & \begin{pmatrix}1+i&0\\1&1\end{pmatrix} \\
\psi_2 & \Lambda_{1,1} & U(2) & \begin{pmatrix}1+i&0\\0&1\end{pmatrix} \\
\psi_2^\prime & \Lambda_2 & A_1\oplus A_1(2) & \begin{pmatrix}1&1+i\\1&1-i\end{pmatrix} \\
\psi_2^\prime & \Lambda_{1,1} & A_1(2)\oplus (2) & \begin{pmatrix}1-i&1\\1+i&1\end{pmatrix} \\
\psi_4 & \Lambda_2^2 & D_4(2) & \begin{pmatrix}-1&i&-i&-i\\0&i&0&-1-i\\-i&1&-1&-1\\0&-i&0&-1+i\end{pmatrix}\\
\bottomrule
\end{array}
\]
\caption{Some antiunitary transformations of Gaussian lattices of small rank.}
\label{fixedblockstable}
\end{table}

\end{lem}
 
\begin{proof}
Using the conditions on $M_i$ from Proposition \ref{antiunit} it is a straightforward calculation to prove that the $\psi_i$ are antiunitary involutions. Furthermore we need to check that the columns of $B^{\psi_i}$ form a basis for the fixed point lattice $\Lambda^{\psi_i}$ and that $\Lambda^{\psi_i}=\overline{B^{\psi_i}}^t\Lambda B^{\psi_i}$. For example the fixed point lattice $(\Lambda_2^2)^{\psi_4}$ is given by the subset \[ \{ (z_1,z_2,i\bar{z}_1,\bar{z}_2) \ ; \ z_1,z_2\in \mathcal{G} \} \subset \Lambda_2^2 \]
and it is not difficult to check that the columns of $B^{\psi_4}$ indeed form a $\mathbb{Z}$-basis. The verification for the other lattices proceeds similarly.   
\end{proof}

\begin{lem}\label{chiandichi}The antiunitary involution $\psi_2$ is conjugate in $U(\Lambda_2)$ to $i\psi_2$, likewise $\psi_2^\prime$ is conjugate in $U(\Lambda_{1,1})$ to $i\psi_2^\prime$, and likewise $\psi_4$ is conjugate in $U(\Lambda_2^2)$ to $i\psi_4$.
\end{lem}
\begin{proof}
The matrix $N=\left( \begin{smallmatrix}0&i\\1&0\end{smallmatrix} \right)$ satisfies $\overline{N}^t\Lambda_2 N =\Lambda_2$ and  $\overline{N}^t\Lambda_{1,1} N =\Lambda_{1,1}$ so it is contained in $U(\Lambda_2)$ and $U(\Lambda_{1,1})$. It also satisfies $N\psi_2\overline{N}^{-1}=i\psi_2$ and  $N\psi_2^\prime \overline{N}^{-1}=i\psi_2^\prime$. Similarly conjugation by the matrix $\left( \begin{smallmatrix}0&N\\N&0\end{smallmatrix} \right) \in U(\Lambda_2^2)$ maps $\psi_4$ to $i\psi_4$. \end{proof}

\subsection{Ball quotients from hyperbolic lattices}

Let $\Lambda$ be a Gaussian lattice of hyperbolic signature $(1,n)$ with $n\geq 2$ such that $\Lambda \subset (1+i)\Lambda^\vee$. We can associate to $\Lambda$ a complex ball:
\[ \mathbb{B} = \mathbb{P}\{ x \in \Lambda \otimes_{\mathcal{G}} \mathbb{C} \ ; \ h(x,x) > 0 \}. \]  

The group $P\Gamma=PU(\Lambda)$ acts properly discontinuously on $\mathbb{B}$. The ball quotient $P\Gamma \backslash \mathbb{B}$ is a quasi-projective variety of finite hyperbolic volume by the theorem of Baily-Borel \cite{BailyBorel}. Recall that a root $r \in \Lambda$ is an element of norm $-2$. We denote by $\mathcal{H}\subset \mathbb{B}$ the union of all the root mirrors $H_r$ and write $\mathbb{B}^\circ = \mathbb{B} - \mathcal{H}$. In all the examples we consider later on the space $P\Gamma \backslash \mathbb{B}^\circ$ is a moduli space for certain smooth objects. The image of $\mathcal{H}$ in this space is called the discriminant and parametrizes certain singular objects. The following lemma describes how two mirrors in $\mathcal{H}$ can intersect.   

\begin{lem}\label{MirrorIntersect}Let $r_1,r_2$ be two roots in $\Lambda$ such that $H_{r_1} \cap H_{r_2} \neq \emptyset$. The projective classes $[ r_1]$ and $[r_2]$ are either identical, orthogonal or they span a Gaussian lattice of type $\Lambda_2$. 
\end{lem}    
\begin{proof}
Since the images of $H_{r_1}$ and $H_{r_2}$ meet in $\mathbb{B}$ there is a vector $x\in \Lambda$ with $h(x,x)>0$ orthogonal to both $r_1$ and $r_2$. This implies that $r_1$ and $r_2$ span a negative definite space so that the Hermitian matrix
\[ \begin{pmatrix}-2&h(r_1,r_2)\\ h(r_2,r_1)&-2 \end{pmatrix} \]
is negative definite. This is equivalent to $\vert h(r_1,r_2)\vert^2 < 4$ and since $h(r_1,r_2)\in (1+i)\mathcal{G}$ we see that either $h(r_1,r_2)=0$ or $h(r_1,r_2)=\pm 1 \pm i$. In the second case we can assume that $h(r_1,r_2)=1+i$ by multiplying $r_1$ and $r_2$ by suitable units in $\mathcal{G}^\ast$.
\end{proof}

 Let $\mathbb{B}^\chi$ be the fixed point set in $\mathbb{B}$ of the real form $[\chi]$. Since the fixed point lattice $\Lambda^\chi$ is of hyperbolic signature this is a real ball given by
\[ \mathbb{B}^\chi = \mathbb{P}\{ x \in \Lambda^\chi \otimes_\mathbb{Z} \mathbb{R} \ ; \ h(x,x)>0 \}. \]
Note that the lattice $\Lambda^{i\chi}$ defines the same real ball. The isomorphism type of the unordered pair $(\Lambda^\chi,\Lambda^{i\chi})$ is an invariant of the $P\Gamma$-conjugacy class  of $[\chi]$ as shown by the following lemma. This invariant will prove very useful to distinguish between classes up to $P\Gamma$-conjugacy.

\begin{lem}\label{ClassCriterion}
If the projective classes $[\chi]$ and $[\chi^\prime]$ of two antiunitary involutions $\chi$ and $\chi^\prime$ of $\Lambda$ are conjugate in $P\Gamma$ then the isomorphism classes of the unordered pairs of lattices $( \Lambda^\chi , \Lambda^{i\chi} )$ and $( \Lambda^{\chi^\prime} , \Lambda^{i\chi^\prime} )$ are equal.
\end{lem}
\begin{proof}
Suppose $[\chi]$ and $[\chi^\prime]$ are conjugate in $P\Gamma$. Then there is a $g\in \Gamma$ such that $[g\chi g^{-1}]=[\chi^\prime]$. This implies that $g\chi g^{-1} = \lambda \chi^\prime$ for some unit $\lambda \in \mathcal{G}^\ast$ so that the antiunitary involutions $\chi$ and $\lambda \chi^\prime$ are conjugate in $\Gamma$. From this we deduce that $\Lambda^\chi \cong \Lambda^{\lambda \chi^\prime}$. Since $g\in \Gamma$ commutes with multiplication by $i$ the involutions $i\chi$ and $i\lambda \chi^\prime$ are also conjugate in $\Gamma$ and we get $\Lambda^{i\chi} \cong \Lambda^{i\lambda \chi^\prime}$.  
\end{proof}

\begin{prop}\label{PGammaChi}Let $P\Gamma^\chi$ be the stabilizer of $\mathbb{B}^\chi$ in $P\Gamma$. Then we have
\[ P\Gamma^\chi = \{ [g] \in P\Gamma \ ; \ [g] \circ [\chi] = [\chi] \circ [g] \} \]
\end{prop}
\begin{proof}
The following statements are equivalent:
\begin{align*}
&[g]\in P\Gamma^\chi ,&\\
&[gx] \in \mathbb{B}^\chi  &\text{ for all } [x]\in \mathbb{B}^\chi ,\\
&[\chi(gx)] = [g(\chi x)] &\text{ for all } [x]\in \mathbb{B}^\chi ,\\
&[\chi(gz)] = [g(\chi z)] &\text{ for all } [z]=[x+iy] \ , \ [x],[y] \in \mathbb{B}^\chi, \\
&[g \circ \chi] = [\chi \circ g].&
\end{align*}
\end{proof}
From Proposition \ref{PGammaChi} we see that for every element $[g]\in P\Gamma^\chi$ precisely one of the following holds:
\begin{enumerate}
\item[I.] There is a $g\in [g]$ such that: $g\chi g^{-1}=\chi$ so that: $g\Lambda^\chi = \Lambda^\chi $.
\item[II.] There is a $g\in [g]$ such that $g\chi g^{-1}=i\chi$ so that: $g\Lambda^\chi = \Lambda^{i\chi}$.
\end{enumerate}
We use Chu's convention from \cite{Chu} and say that $[g]\in P\Gamma^\chi$ is of type $I$ or of  type $II$ respectively. The elements of type $I$ form a subgroup of $P\Gamma^\chi$ which we denote by $P\Gamma^\chi_I$. If there exists an element of type $II$ then this subgroup is of index $2$, otherwise every element of $P\Gamma^\chi$ is of type $I$. 

Every element $[g]\in P\Gamma^\chi$ of type $I$ determines a unique element in $PO(\Lambda^\chi)$ so there is a natural embedding $P\Gamma^\chi_I \hookrightarrow PO(\Lambda^\chi)$. In general not every element $[g]\in PO(\Lambda^\chi)$ extends to the group $P\Gamma$. Let $B^\chi$ be a matrix whose columns represent a basis for the lattice $\Lambda^\chi$ in $\Lambda$. Then we have
\begin{equation}\label{realcondition}P\Gamma^\chi_I = \{ [M]\in PO(\Lambda^\chi) \ ; \ B^\chi M(B^\chi)^{-1} \in \Mat_{n+1}(\mathcal{G}) \}
\end{equation}
so that $P\Gamma^\chi_I$ is the subgroup of $PO(\Lambda^\chi)$ consisting of all elements that extend to unitary transformations of the Gaussian lattice $\Lambda$. 

\begin{thm}
The groups $P\Gamma^\chi$ and $PO(\Lambda^\chi)$ are commensurable.\end{thm} 

\begin{proof}We have seen that the intersection of the two groups is given by
\[ P\Gamma^\chi \cap PO(\Lambda^\chi) = P\Gamma^\chi_I\]
and has at most index $2$ in $P\Gamma^\chi$. We now prove that this intersection is a congruence subgroup of $PO(\Lambda^\chi)$ so that in particular it has finite index. Recall that the adjoint matrix $B^{\adj}$ has coefficients in $\mathcal{G}$ and satisfies $(\det B)B^{-1}=B^{\adj}$. If we write $M=1+X$ then by Equation \ref{realcondition} we have $[M]\in P\Gamma^\chi_I$ if and only if $\det B$ divides $BXB^{\adj}$. This is certainly the case if $\det B$ divides $X$ so if $M \equiv 1 \bmod{(\det B)}$. This implies that $P\Gamma^\chi_I$ contains the principal congruence subgroup 
\[ \{ [M]\in PO(\Lambda^\chi) \ ; \ M \equiv 1 \bmod{(\det B)} \}. \]
\end{proof}

In the examples we encounter the lattice $\Lambda^\chi$ is reflective so that the reflections generate a finite index subgroup in $PO(\Lambda^\chi)$. By the results of Section \ref{Vinberg} the group $PO(\Lambda^\chi)$ is of the form $W(C)\rtimes S(C)$ where $C\subset \mathbb{B}^\chi$ is a Coxeter polytope  of finite volume, $W(C)$ its reflection group and $S(C)$ a group of automorphisms of $C$. The polytope $C$ can be determined by Vinberg's algorithm. We will see  that in many cases the reflection subgroup of the group $P\Gamma^\chi_I$ is also of finite index. This can be determined by  applying Vinberg's algorithm with the condition that in every step we only accept roots $r$ such that the reflection $s_r \in PO(\Lambda^\chi)$ satisfies Equation \ref{realcondition}. This is equivalent to the condition
\begin{equation}\label{realrootcondition}
\frac{2r}{h(r,r)} \in \Lambda^\vee.
\end{equation}

We finish this section by describing how a root mirror $H_r \in \mathcal{H}$ can meet the real ball $\mathbb{B}^\chi$. This intersection can be of codimension one or two as shown by the following lemma.

\begin{lem}
Suppose $r \in \Lambda$ is a root such that $\mathbb{B}^\chi \cap H_r \neq \emptyset$. Then $\mathbb{B}^\chi \cap H_r$ is equal to $\mathbb{B}^\chi \cap L$ with $L^\perp$ a lattice in $\Lambda^\chi$ of type $A_1,A_1(2),A_1\oplus A_1(2)$ or $A_1(2)^2$.
\end{lem}

\begin{proof}
If $x\in H_r \cap \mathbb{B}^\chi$ then $x$ is fixed by both $s_r$ and $s_{\chi r}$ so the intersection $H_r \cap H_{\chi r}$ is nonempty and we are in the situation of Lemma \ref{MirrorIntersect}. Suppose that $\chi [r] =[r]$. If $\chi r = \pm r$  then either $r$ or $ir$ is a root of $\Lambda^\chi$. Both have norm $-2$ so they span a root system of type $A_1$. If $\chi r = \pm i r$ then one of $(1\pm i)r$ is a root of $\Lambda^\chi$. Both have norm $-4$ so they span a root system of type $A_1(2)$. If $\chi [r] \neq [r]$ then the roots $r$ and $\chi r$ span a rank two Gaussian lattice that is either $(-2)\oplus (-2)$ or $\Lambda_2$ according to Lemma \ref{MirrorIntersect}. The involution $\chi$ acts on these lattices as the antiunitary involution $\psi_2^\prime$. The fixed point lattice for $(-2)\oplus (-2)$ is $A_1(2)^2$ as follows from a straightforward computation. For $\Lambda_2$ we get the fixed point lattice $A_1\oplus A_1(2)$ as follows from Lemma \ref{fixedblocks}.\end{proof}

\subsection{Examples}

\subsubsection*{The Gaussian lattice $\Lambda_{1,2}$}\label{gamma12}

The lattice $\Lambda_{1,2}=\Lambda_2\oplus (2)$ of signature $(1,2)$ is related to the moduli space $\mathcal{M}(321^3)$ of eight-tuples of points on $\mathbb{P}^1$ such that there are unique points of multiplicity $3$ and $2$ and three distinct points of multiplicity $1$. We study the antiunitary involutions of this lattice in some detail. Using Table \ref{fixedblockstable} we can immediately write down two antiunitary involutions of $\Lambda_{1,2}$, namely $\psi_2\oplus \psi_1$ and $\psi_2^\prime \oplus \psi_1$. We will prove that their projective classes are distinct modulo conjugation in $P\Gamma=PU(\Lambda_{1,2})$. There is however a another antiunitary involution of $\Lambda_{1,2}$ given by $\psi_3 =M_3\circ \conj$ where $M_3$ is the complicated matrix
\[ M_3=\begin{pmatrix}-2+i&2-2i&-2-2i\\2&-1&2i\\1+3i&-2-2i&-3+2i\end{pmatrix}.\] 
This antiunitary involution takes on a much simpler form if we change to a different basis for $\Lambda_{1,2}$ as shown by the following lemma.

\begin{lem}\label{Gausiso}
The Gaussian lattices $\Lambda_2\oplus (2)$ and $(-2)\oplus \Lambda_{1,1}$ are isomorphic. The antiunitary involution $\psi_3$ of $\Lambda_2\oplus (2)$ maps to the antiunitary involution $\psi_1\oplus \psi_2$ of $(-2)\oplus \Lambda_{1,1}$ under this isomorphism. 
\end{lem}
\begin{proof}The underlying $\mathbb{Z}$-lattices of the Gaussian lattices $\Lambda_2\oplus (2)$ and $(-2)\oplus \Lambda_{1,1}$ are $D_4\oplus (2)$ and $U\oplus U(2)\oplus A_1$. Both are even $2$-elementary lattices and the invariants $(r_+,r_-,l,\delta)$ of Theorem \ref{2invariants} are easily seen to be $(1,2,3,1)$ for both lattices hence they are isomorphic. An explicit base change is given by $\overline{B}^t(\Lambda_2\oplus (2))B=(-2)\oplus \Lambda_{1,1}$ for the unimodular matrix
\[ B = \begin{pmatrix}1+i&i&0\\1-i&0&1\\1&1&i\end{pmatrix} . \]
The final statement follows from the equality $B(\psi_1\oplus \psi_2)\overline{B}^{-1}=\psi_3$.
\end{proof}

\begin{prop}The projective classes of the three antiunitary involutions $\chi$ given by $\psi_2\oplus (2), \ \psi_2^\prime \oplus (2)$ and $\psi_3$ of $\Lambda_{1,2}$ are distinct modulo conjugation in $P\Gamma$. The groups $P\Gamma^\chi$ of these involutions are hyperbolic Coxeter groups and their Coxeter diagrams are shown in Table \ref{threeclasses}. \end{prop}

\begin{proof}
We will use Lemma \ref{ClassCriterion} to show that the projective classes of the three antiunitary involutions are not $P\Gamma$-conjugate. For this we need to calculate the fixed point lattices of $\chi$ and $i\chi$ for all three antiunitary involutions. These can be read off from Table \ref{fixedblockstable} for the antiunitary involutions $\psi_2\oplus \psi_1$ and $\psi_2^\prime \oplus \psi_1$. For $\psi_3$ we use Lemma \ref{Gausiso} combined with Table \ref{fixedblockstable}. We also use Lemma \ref{simplifylattices} to simplify the lattices. For example one has
\begin{align*}
\Lambda_{1,2}^{i({\psi_2\oplus \psi_1})} &\cong (4) \oplus A_1^2 \\
&\cong (2) \oplus A_1 \oplus A_1(2)
\end{align*}
where the first isomorphism follows from Table \ref{fixedblockstable} and Lemma \ref{chiandichi} and the second follows from Lemma \ref{simplifylattices}. The results are listed in Table \ref{threeclasses}. The lattices $(2) \oplus A_1\oplus A_1(2)$ and $U(2) \oplus A_1(2)$ in this table are not isomorphic. Indeed, if we scale them by a factor $\frac{1}{2}$ then one is even while the other is not. This proves that the $P\Gamma$-conjugation classes of $\psi_2\oplus \psi_1$ and $\psi_3$ are distinct. We can distinguish the fixed point lattices of $\psi_2^\prime \oplus \psi_1$ from the previous two by calculating their discriminants.

\begin{table}
\centering
\[
\begin{array}{lll}
\toprule
\chi & \Lambda_{1,2}^\chi & \Lambda_{1,2}^{i\chi} \\
\midrule
\psi_2\oplus \psi_1 & (2)\oplus A_1^2 & (2) \oplus A_1 \oplus A_1(2) \\
\psi_3 & (2) \oplus  A_1^2  & U(2) \oplus A_1(2)  \\
\psi_2^\prime \oplus \psi_1 & (2) \oplus A_1 \oplus A_1(2) & (2) \oplus A_1(2)^2 \\
\bottomrule
\end{array}
\]
\caption{The three classes of antiunitary involutions of the lattice $\Lambda_{1,2}$. }
\label{threeclasses}
\end{table}
\end{proof}

The moduli space $\mathcal{M}(321^3)$ has three connected components so the three projective classes actually form a complete set of representatives for $P\Gamma$-conjugation classes of antiunitary involutions in $\Lambda_{1,2}$. For more information we refer to \cite{Rieken} Section 3.5.

\subsubsection*{The Gaussian lattice $\Lambda_{1,6}$} \label{Lambda16}

The lattice $\Lambda_{1,6} = \Lambda_2^3\oplus (2)$ is related to the moduli space of plane quartic curves. In this section we collect some useful properties of this lattice that will be used in later sections. We start by introducing a very convenient basis.\begin{lem}There is a basis $\{e_1,\ldots,e_7\}$ for $\Lambda_{1,6}$ so that the basis vectors are enumerated by the vertices of the Coxeter diagram of type $E_7$ as in Figure \ref{VinCox}. By this we mean that the basis satisfies
\[
h(e_i,e_j) = \begin{cases}
-2 & \text{if } i=j \\
1+\sign(j-i)i & \text{if } $i,j$ \text{ connected} \\
0 & \text{else}.
\end{cases}
\]
\end{lem}
\begin{proof}
An example of such a basis is given by the column vectors of the matrix
\[ B_{E_7} =
\begin{pmatrix}
1&-1-i&0&0&0&0&0\\
0&-1&0&0&0&0&0\\
0&-1-i&1&-1-i&0&0&0\\
0&-1&0&-1&0&0&1\\
0&0&0&-1-i&1&0&0\\
0&0&0&-1&0&1&0\\
0&1&0&1&0&0&0
\end{pmatrix}.
\]
\end{proof}
The tetraflections $t_{e_i}\in U(\Lambda_{1,6})$ with $i=1,\ldots,7$ satisfy the commutation and braid relations of the Artin group $A(E_7)$ of type $E_7$ so that they induce a representation $A(E_7)\rightarrow U(\Lambda_{1,6})$ by tetraflections. In fact this homomorphism extends to an epimorphism $A(\widetilde{E}_7) \rightarrow U(\Lambda_{1,6})$ as follows from \cite{HeckmanLooijenga,Looijenga2}. Hence the lattice $\Lambda_{1,6}$ is tetraflective.

\begin{prop}\label{reduction}Let $V$ be the orthogonal vectorspace over $\mathbb{F}_2$ defined by 
\[ V = \Lambda_{1,6}/(1+i)\Lambda_{1,6} \cong (\mathbb{F}_2)^7 \]
with the invariant quadratic form $q(x) \equiv \frac{1}{2}h(x,x) \bmod{2}$.  
Reduction modulo $(1+i)$ induces a surjective homomorphism
\[ U(\Lambda_{1,6}) \rightarrow O(V,q) \cong W(E_7)^+. \]
where we denote by $W(E_7)^+$ the Weyl group of type $E_7$ divided modulo its center $\{\pm 1\}$. This group is generated by the images of the tetraflections $t_{e_i}$ with $i=1,\ldots,7$.
\end{prop} 
\begin{proof}
The tetraflections $t_{e_i}$ with $i=1,\dots,7$ act as reflections on the vectorspace $V$ since their squares act as the identity. This defines a representation of the Weyl group $W(E_7)$ on $V$. The matrices of these tetraflections modulo $(1+i)$ are identical to the matrices of the simple generating reflections of $W(E_7)$ modulo $2$. These act naturally on the $\mathbb{F}_2$-vectorspace $V^\prime = Q/2Q$ where $Q$ is the root lattice of type $E_7$.  This space is equipped with the invariant quadratic form defined by $q^\prime(x) \equiv \frac{1}{2}(x,x) \bmod{2}$ where $(\cdot,\cdot)$ is the natural bilinear form on $Q$ defined by the Gram matrix of type $E_7$. We conclude that the representation spaces $(V,q)$ and $(V^\prime,q^\prime)$ for $W(E_7)$ are isomorphic. The proposition now follows from Exercise 3 in \S 4 of Ch VI of \cite{Bourbaki} where it is shown that
\[ 1 \rightarrow \{ \pm 1 \} \rightarrow W(E_7) \rightarrow O(V,q) \rightarrow 1\]
is an exact sequence.
\end{proof}

Let $U(\Lambda_{1,6})^a$ be the set of antiunitary transformations of $\Lambda_{1,6}$. Reduction modulo $(1+i)$ also induces a map 
\[ U(\Lambda_{1,6})^a \rightarrow O(V,q)\cong W(E_7)^+\]
since complex conjugation induces the identity map on $V$. The projective class of an antiunitary involution $[\chi]$ maps to an involution $u$ of $W(E_7)^+$ under this map. Its image does not depend on the choice of representative for the class $[\chi]$ since multiplication by $i$ acts as the identity on $V$. This implies that the conjugation class of the involution $u$ in $W(E_7)^+$ is an invariant of the $P\Gamma$-conjugation class of $[\chi]$. The conjugation classes of involutions of $W(E_7)$ are determined by Wall\cite{Wall}. This can also be derived from more general results by Richardson \cite{Richardson}. We review these results in the Appendix of this article. There are ten conjugation classes that come in five pairs $\{u,-u\}$. Since both $u$ and $-u$ map to the same involution $\overline{u} \in W(E_7)^+$ each pair determines a unique conjugation class in $W(E_7)^+$. We will use this to prove the following theorem.

\begin{thm}\label{sixinvolutions}If we define (in the notation of Table \ref{fixedblockstable} and Lemma \ref{Gausiso}) antiunitary involutions $\chi_i$ of $\Lambda_2^3\oplus (2)$ for $i=1,\ldots,6$ by
\begin{align*}
\chi_1 &=\psi_2^3 \oplus \psi_1 &
\chi_2 &=\psi_2^2 \oplus \psi_2^\prime \oplus \psi_1 &
\chi_3 &=\psi_2 \oplus (\psi_2^\prime)^2 \oplus \psi_1 \\
\chi_4 &= (\psi_2^\prime)^3 \oplus \psi_1 &
\chi_5 &= \psi_4 \oplus \psi_2 \oplus \psi_1 &
\chi_6 &= \psi_4 \oplus \psi_3
\end{align*}
then their projective classes are distinct modulo conjugation by $P\Gamma$.
\end{thm}
\begin{proof}According to Lemma \ref{fixedblocks} and the previous example it is clear that the $\chi_i$ are antiunitary involutions of the lattice $\Lambda_2^3\oplus (2)$. By reducing the $\chi_i$ modulo $(1+i)$ they map to involutions $\overline{u}_i$ in $W(E_7)^+$. To distinguish them we calculate the dimensions of the fixed point spaces in $V$ and compare them to those of the involutions in $W(E_7)^+$. From this we conclude that $\overline{u}_1,\overline{u}_2$ and $\overline{u}_4$ are of type $(1,E_7),(A_1,D_6)$ and $(A_1^3,A_1^4)$ respectively. It is clear that $\overline{u}_5=\overline{u}_6$. We used the computer algebra package SAGE to determine that both are of type $(D_4,A_1^{3\prime})$ and that $\overline{u}_3$ is of type $(A_1^2,D_4A_1)$. All of this is summarized in Table \ref{dimfix}.

\begin{table}
\centering
\[
\begin{array}{lll}
\toprule
\chi_i & \text{Type of } \overline{u}_i & \dim V^{\overline{u}_i} \\
\midrule
\chi_1 & (1,E_7) & 7 \\
\chi_2 & (A_1,D_6) & 6 \\
\chi_3 & (A_1^2,D_4A_1) &  5 \\
\chi_4 & (A_1^3,A_1^4) & 4 \\
\chi_5 & (D_4,A_1^{3\prime}) & 5 \\
\chi_6 & (D_4,A_1^{3\prime}) & 5 \\
\bottomrule
\end{array}
\]
\caption{The six projective classes of antiunitary involutions of $\Lambda_{1,6}$ and the type of the involution they induce in $W(E_7)^+$ by reducing modulo $1+i$.}
\label{dimfix}
\end{table}

This method is insufficient to distinghuish the classes of $\chi_5$ and $\chi_6$. For this we determine the fixed point lattice $\Lambda_{1,6}^{\chi_i}$ and $\Lambda_{1,6}^{i\chi_i}$ for $i=5,6$ and use Lemma \ref{ClassCriterion}. The lattices $\Lambda_{1,6}^{\chi_5}$ and $\Lambda_{1,6}^{\chi_6}$ are both isomorphic to $(2) \oplus A_1^2 \oplus D_4(2)$. The lattice $\Lambda_{1,6}^{i\chi_5}$ is isomorphic to 
\[(2) \oplus A_1 \oplus A_1(2) \oplus D_4(2) \cong (2) \oplus A_1^3 \oplus A_1(2)^3 \] 
where we used Lemma \ref{simplifylattices}. The fixed point lattice $\Lambda_{1,6}^{i\chi_6}$ is isomorphic to $U(2) \oplus A_1(2) \oplus D_4(2)$. After scaling by a factor $\frac{1}{2}$ we see that $\Lambda_{1,6}^{i\chi_5}$ is odd while the $\Lambda_{1,6}^{i\chi_6}$ is even so that they are not isomorphic. Consequently the $P\Gamma$-conjugacy classes of the $[\chi_5]$ and $[\chi_6]$ are distinct.
\end{proof}
\begin{table}
\centering
\[ \begin{array}{lllll}
\toprule
\chi_i & \Lambda_{1,6}^{\chi_i} & d(\Lambda_{1,6}^{\chi_i}) & \Lambda_{1,6}^{i\chi_i} & d(\Lambda_{1,6}^{i\chi_i}) \\
\midrule
\chi_1 & (2) \oplus A_1^6  & 2^7 & (2) \oplus A_1^5 \oplus A_1(2) & 2^8 \\
\chi_2 & (2) \oplus A_1^5 \oplus A_1(2)  & 2^8 & (2) \oplus  A_1^4 \oplus A_1(2)^2 & 2^9 \\
\chi_3 & (2) \oplus  A_1^4 \oplus A_1(2)^2  & 2^9 & (2) \oplus A_1^3 \oplus A_1(2)^3  & 2^{10} \\
\chi_4 & (2) \oplus  A_1^3 \oplus A_1(2)^3  & 2^{10} & (2) \oplus  A_1^2 \oplus A_1(2)^4   & 2^{11} \\
\chi_5 & (2) \oplus A_1^2  \oplus D_4(2)  & 2^9 &(2) \oplus A_1^3 \oplus A_1(2)^3  & 2^{10} \\
\chi_6 & (2) \oplus A_1^2  \oplus D_4(2)   & 2^9 & U(2) \oplus A_1(2) \oplus D_4(2) & 2^{10}\\
\bottomrule
\end{array} \]
\caption{The fixed point lattices for $\chi_j$ and $i\chi_j$ for $j=1,\ldots,6$ and their discriminants.}
\label{12lattices}
\end{table}

\begin{rmk}The question remains whether the list of antiunitary involutions from Theorem \ref{sixinvolutions} is complete. This is in fact the case as we will see in Proposition \ref{allsixinvolutions}. It is a consequence of the fact that the moduli space of smooth real quartics consists of six connected components. 
\end{rmk}

\begin{thm}The hyperbolic lattices $\Lambda_{1,6}^\chi$ for $\chi=\chi_j,i\chi_j$ where $j=1,\ldots,6$ from Table \ref{12lattices} are all reflective and the hyperbolic Coxeter diagrams for the groups $PO(\Lambda_{1,6}^\chi)$ are shown in Figure \ref{E7diagrams}.
\end{thm}

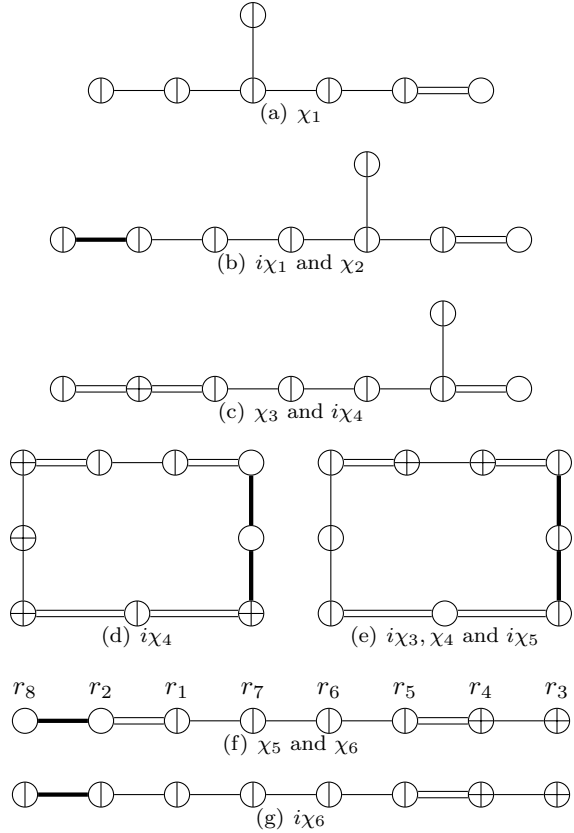
\begin{figure}[h]
\centering
\begin{tabular}{c}
\subfigure[$\chi_1$]{
\begin{tikzpicture}
\node[c,vert] (1) at (0,0) {};
\node[c,vert] (2) at (1,0) {};
\node[c,vert] (3) at (2,0) {};
\node[c,vert] (4) at (3,0) {};
\node[c,vert] (5) at (4,0) {};
\node[c] (6) at (5,0) {};
\node[c,vert] (7) at (2,1) {};
\draw [-] (1) -- (2) -- (3) -- (4) -- (5);
\draw [-,double distance = 2pt] (5) -- (6);
\draw [-] (3) -- (7);
\end{tikzpicture}
}

\\

\subfigure[$i\chi_1$ and $\chi_2$]{
\begin{tikzpicture}
\node[c,vert] (1) at (0,0) {};
\node[c,vert] (2) at (1,0) {};
\node[c,vert] (3) at (2,0) {};
\node[c,vert] (4) at (3,0) {};
\node[c,vert] (5) at (4,0) {};
\node[c,vert] (6) at (5,0) {};
\node[c] (7) at (6,0) {};
\node[c,vert] (8) at (4,1) {};
\draw[-,ultra thick] (1) -- (2);
\draw[-] (2) -- (3) -- (4) -- (5) -- (6) ;
\draw[-,double distance = 2pt] (6) -- (7);
\draw[-] (5) -- (8);
\end{tikzpicture}
}

\\

\subfigure[$\chi_3$ and $i\chi_4$]{
\begin{tikzpicture}
\node[c,vert] (1) at (0,0) {};
\node[c,oplus] (2) at (1,0) {};
\node[c,vert] (3) at (2,0) {};
\node[c,vert] (4) at (3,0) {};
\node[c,vert] (5) at (4,0) {};
\node[c,vert] (6) at (5,0) {};
\node[c] (7) at (6,0) {};
\node[c,vert] (8) at (5,1) {};
\draw[-,double distance = 2pt] (1) -- (2) -- (3);
\draw[-] (3) -- (4) -- (5) -- (6) ;
\draw[-,double distance = 2pt] (6) -- (7);
\draw[-] (6) -- (8);
\end{tikzpicture}
}

\\

\subfigure[$i\chi_4$]{
\begin{tikzpicture}
\node[c,oplus] (1) at (0,0) {};
\node[c,vert] (2) at (1.5,0) {};
\node[c,oplus] (3) at (3,0) {};
\node[c] (4) at (3,1) {};
\node[c] (5) at (3,2) {};
\node[c,vert] (6) at (2,2) {};
\node[c,vert] (7) at (1,2) {};
\node[c,oplus] (8) at (0,2) {};
\node[c,oplus] (9) at (0,1) {};
\draw[-,double distance = 2pt] (1) -- (2) -- (3);
\draw[-,ultra thick] (3) -- (4) -- (5);
\draw[-,double distance = 2pt] (5) -- (6);
\draw[-] (6) -- (7);
\draw[-,double distance = 2pt] (7) -- (8);
\draw[-] (8) -- (9) -- (1);
\end{tikzpicture}
}

\quad
\subfigure[$i\chi_3,\chi_4$ and $i\chi_5$]{
\begin{tikzpicture}
\node[c,vert] (1) at (0,0) {};
\node[c] (2) at (1.5,0) {};
\node[c,vert] (3) at (3,0) {};
\node[c,vert] (4) at (3,1) {};
\node[c,vert] (5) at (3,2) {};
\node[c,oplus] (6) at (2,2) {};
\node[c,oplus] (7) at (1,2) {};
\node[c,vert] (8) at (0,2) {};
\node[c,vert] (9) at (0,1) {};
\draw[-,double distance = 2pt] (1) -- (2) -- (3);
\draw[-,ultra thick] (3) -- (4) -- (5);
\draw[-,double distance = 2pt] (5) -- (6);
\draw[-] (6) -- (7);
\draw[-,double distance = 2pt] (7) -- (8);
\draw[-] (8) -- (9) -- (1);
\end{tikzpicture}
}

\\
\subfigure[$\chi_5$ and $\chi_6$]{
\begin{tikzpicture}
\node[c,label = above:$r_8$] (0) at (-1,0) {};
\node[c,label = above:$r_2$] (1) at (0,0) {};
\node[c,vert,label = above:$r_1$] (2) at (1,0) {};
\node[c,vert,label = above:$r_7$] (3) at (2,0) {};
\node[c,vert,label = above:$r_6$] (4) at (3,0) {};
\node[c,vert,label = above:$r_5$] (5) at (4,0) {};
\node[c,oplus,label = above:$r_4$] (6) at (5,0) {};
\node[c,oplus,label = above:$r_3$] (7) at (6,0) {};
\draw[-,double distance = 2pt] (1) -- (2);
\draw[-] (2) -- (3) -- (4) -- (5) ;
\draw[-,double distance = 2pt] (5) -- (6);
\draw[-] (6) -- (7);
\draw[-,ultra thick] (0) --  (1);
\end{tikzpicture}
}

\\

\subfigure[$i\chi_6$]{
\begin{tikzpicture}
\node[c,vert] (0) at (-1,0) {};
\node[c,vert] (1) at (0,0) {};
\node[c,vert] (2) at (1,0) {};
\node[c,vert] (3) at (2,0) {};
\node[c,vert] (4) at (3,0) {};
\node[c,vert] (5) at (4,0) {};
\node[c,oplus] (6) at (5,0) {};
\node[c,oplus] (7) at (6,0) {};
\draw[-] (1) -- (2);
\draw[-] (2) -- (3) -- (4) -- (5) ;
\draw[-,double distance = 2pt] (5) -- (6);
\draw[-] (6) -- (7);
\draw[-,ultra thick] (0) -- (1);
\end{tikzpicture}
}
\end{tabular}
\caption{The Coxeter diagram of the groups $PO(\Lambda_{1,6}^{\chi})$ for $\chi=\chi_j,i\chi_j$ with $j=1,\ldots,6$.}
\label{E7diagrams}
\end{figure}  

\begin{proof}
We observe from Table \ref{12lattices} that there are seven distinct hyperbolic lattices. To prove that they are reflective we apply Vinberg's algorithm. We demonstrate this for the hyperbolic lattice $(2)\oplus A_1^2\oplus D_4(2)$ corresponding to the antiunitary involutions $\chi_5$ and $\chi_6$. Let $\{e_0,e_1,e_2\}$ be an orthonormal basis for $(2)\oplus A_1^2$. Recall that the root lattice $D_4(2)$ is given by
\[ D_4(2) = \{ (u_1,u_2,u_3,u_4)\in \mathbb{Z}^4 \ ; \ \sum_{i=0}^4 u_i \equiv 0 \pmod {-2}\}. \]
It contains roots of norm $-4$ and $-8$ and both form a root system of type $D_4$. Together these roots form a root system of type $F_4$. If we choose the controlling vector $e_0$ the height $0$ root system is of type $B_2F_4$ spanned by the roots $\{r_1,\ldots,r_6\}$ from Table \ref{VinbergF}. This table also shows how the algorithm proceeds. The resulting Coxeter diagram is shown in Figure \ref{E7diagrams}. The Coxeter diagrams for the other six hyperbolic lattices can be computed similarly and are also shown in this figure.\end{proof}


\begin{table}
\centering

$\begin{array}{cccccccc}
\toprule
 & e_0 & e_1 & e_2 & u_1 & u_2 & u_3 & u_4 \\
 \midrule
 p & 1 & 0 & 0 & 0& 0&0 &0\\
 \text{height 0} &&&&&&&\\
 r_1 & 0&1&-1&0&0&0&0\\ 
 r_2 & 0&0&1&0&0&0&0\\
 r_3 & 0&0&0&1&-1&-1&-1\\
 r_4 & 0&0&0&0&0&0&2\\
 r_5 & 0&0&0&0&0&1&-1\\
 r_6 & 0&0&0&0&1&-1&0\\
 \text{height} 1 &&&&&&&\\
 r_7 & 1&-1&0&-1&-1&0&0\\
 \text{height} 2 &&&&&&&\\
 r_8 & 1&-1&-1&0&0&0&0\\
\bottomrule
\end{array}$
\caption{Vinberg's algorithm for the hyperbolic lattice $(2)\oplus A_1^2\oplus D_4(2)$. This lattice corresponds to the two antiunitary involutions $\chi_5$ and $\chi_6$.}
\label{VinbergF}
\end{table}

\section{Real plane quartic curves}

\subsection{Kondo's period map} 
In this section we review Kondo's construction of a period map for complex plane quartic curves \cite{Kondo1}. Let $C$ be a smooth quartic curve in $\mathbb{P}^2$ defined by a homogeneous polynomial $f(x,y,z)$ of degree four. We briefly recall some terminology from Mumford's geometric invariant theory of quartic curves \cite{Mumford}. A complex quartic curve is called stable if it has at worst ordinary nodes and cusps as singularities and semistable if it has at worst tacnodes as singularities or is a smooth conic of multiplicity two. 

We define the surface $X$ to be the fourfold cyclic cover of $\mathbb{P}^2$ ramified along $C$ so that
\[X = \{ w^4 = f(x,y,z) \} \subset \mathbb{P}^3.\]
The surface $X$ is a $K3$-surface of degree four with an action of the group of covering transformations of the cover $\pi:X\rightarrow \mathbb{P}^2$. This group is cyclic of order four and a generator is given by the transformation
\[ \rho_X \cdot [w:x:y:z]=[i w : x: y: z].\]
The involution $\tau_X=\rho_X^2$ also acts on $X$ and the quotient surface $Y=X/\tau_X$ is a double cover of $\mathbb{P}^2$ ramified over the quartic $C$. It is a del Pezzo surface of degree two. The situation is summarized by the following commutative diagram.
\begin{center}
\begin{tikzcd}[row sep = small]
{} & X \arrow{dl}{\pi_1} \arrow{dd}{\pi} \\
Y \arrow{dr}{\pi_2} & {} \\
{} & \mathbb{P}^2 
\end{tikzcd}
\end{center}

The cohomology group $H^2(X,\mathbb{Z})$ is even, unimodular of signature $(3,19)$ and so is isomorphic to the $K3$ lattice $L=E_8^2\oplus U^3$. A choice of isomorphism $\phi: H^2(X,\mathbb{Z}) \rightarrow L$ is called a marking of $X$. We fix a marking and let $\rho$ and $\tau$ denote the automorphisms of $L$ induced by $\rho_X$ and $\tau_X$. Kondo \cite{Kondo1} proves that the eigenlattices of $\tau$ for the eigenvalues $+1$ and $-1$ are isomorphic to
\begin{equation}\label{L+andL-}
L_+ \cong  A_1^7 \oplus (2) \quad , \quad L_- \cong D_4^3 \oplus (2)^2 .
\end{equation}

\begin{rmk}
The expression for $L_-$ in Equation \ref{L+andL-} is different from the lattice $U(2)^2 \oplus D_8 \oplus A_1^2$ given by Kondo. Since the lattice $L_-$ is even and $2$-elementary its isomorphism type is determined by the invariants $(r_+,r_-,a,\delta)$ from Theorem \ref{2invariants}. These invariants are $(2,12,8,1)$ for both lattices so that the lattices are isomorphic. For the lattice $U\oplus U(2) \oplus D_4^2 \oplus A_1^2$ the invariants also take these values so that it is isomorphic to the previous two lattices. \end{rmk}

For applications later on it is convenient to have a more explicit desciption of the involution $\tau$. This is provided by the following lemma. 
\begin{lem}\label{explicitembedding}
Let $L=U^3\oplus E_8^2$ be the $K3$ lattice. The involution $\tau$ is conjugate in $O(L)$ to the involution given by
\begin{equation}\label{tau}
-I_2 \oplus \begin{pmatrix}0&I_2\\I_2&0\end{pmatrix} \oplus u \oplus u 
\end{equation} 
where $u\in O(E_8)$ is an involution of type $D_4A_1$.
\end{lem}
\begin{proof}
Since the involution $u$ is of type $D_4A_1$, its negative $-u$ is of type $A_1^3$. This implies that the eigenlattice for the eigenvalue $1$ of $u$ in $E_8$ is isomorphic to $A_1^3$. The $\pm 1$ eigenlattices in $L$ of the involution in Equation \ref{tau} are then given by
\begin{equation}
\begin{aligned}
U(2) \oplus A_1^6 &\cong (2) \oplus A_1^7 \\
U \oplus U(2) \oplus D_4^2 \oplus A_1^2 &\cong D_4^3 \oplus (2)^2 
\end{aligned}
\end{equation}
These eigenlattices are isomorphic to those of $\tau$ in Equation \ref{L+andL-}. The lattice $(2)\oplus A_1^7$ has a unique embedding into the $K3$ lattice $L$ up to automorphisms in $O(L)$ by Theorem \ref{primembed}. This implies that the involution of Equation \ref{tau} is conjugate to $\tau$ in $O(L)$. 
\end{proof}

The map $\pi_1$ induces a primitive embedding of lattices $\pi_1^\ast:\Pic Y \rightarrow \Pic X$ and the image is precisely the lattice $\phi^{-1}(L_+)$. It is the Picard group of the del Pezzo surface $Y$ scaled by a factor two which comes from the fact that the map $\pi_1$ is of degree two. 

The powers $\rho,\rho^2$ and $\rho^3$ act on the lattice $L_-$ without fixed points. This action turns $L_-$ into a Gaussian lattice of signature $(1,6)$ isomorphic to the Gaussian lattice $\Lambda_{1,6}=\Lambda_2^3\oplus (2)$. From now on we identify $L_-$ considered as a Gaussian lattice with $\Lambda_{1,6}$ and write $L_-$ for the underlying $\mathbb{Z}$-lattice. If $\gamma \in \pi_1^\ast \Pic(Y)$ then $(\omega,\gamma)=0$ for all $\omega \in H^{2,0}(X,\mathbb{C})$ so that the complex ball:
\[ \mathbb{B} = \mathbb{P}\{ x \in \Lambda_{1,6} \otimes_\mathcal{G} \mathbb{C} \ ; \ h(x,x)>0 \} \]
is a period domain for smooth plane quartic curves. Let $\Gamma=U(\Lambda_{1,6})$ be the unitary group of the Gaussian lattice $\Lambda_{1,6}$. Equivalently it is the group of orthogonal transformations of the lattice $L_-$ that commute with $\rho$. The period map $\Per: \mathcal{Q} \rightarrow P\Gamma \backslash \mathbb{B}$ is injective by the Torelli theorem for $K3$ surfaces but not surjective. Its image misses certain divisors in $\mathbb{B}$ which we now describe. An element $r \in \Lambda_{1,6}$ is called a root if $h(r,r)=-2$ and for every root we define the mirror $H_r = \{ z\in \mathbb{B} \ ; \ h(r,z)=0 \}$. We denote by $\mathcal{H}\subset \mathbb{B}$ the union of all the root mirrors $H_r$ and write $\mathbb{B}^\circ = \mathbb{B} \setminus \mathcal{H}$.

\begin{thm}[Kondo]\label{Kondomain}
The period map defines an isomorphism of holomorphic orbifolds
\[ \Per: \mathcal{Q} \rightarrow P \Gamma \backslash \mathbb{B}^\circ.\]
\end{thm}

\begin{proof}
The proof consists of constructing an inverse map of the period map. We give a brief sketch of the main arguments used in \cite{Kondo1}. Let $z\in \mathbb{B}^\circ$. There is a $K3$ surface $X$ together with a marking $\phi:H^2(X,\mathbb{Z})\rightarrow L$ such that the period point of $X$ is $z$. This  $K3$ surface $X$ has an automorphism $\rho_X$ of order four such that its action on $H^2(X,\mathbb{Z})$ corresponds to the action of $\rho$ on $L$. The quotient surface $Y=X/\left< \tau_X \right>$ with $\tau_X = \rho_X^2$ is a del Pezzo surface of degree two. Its anticanonical map: $|K_Y|:Y\rightarrow \mathbb{P}^2$ is a double cover of $\mathbb{P}^2$ ramified over a smooth plane quartic curve $C$. The inverse period map associates to the $P\Gamma$-orbit of $z\in \mathbb{B}^\circ$ the isomorphism class of this quartic curve $C$.
\end{proof}

Furthermore Kondo proves in \cite{Kondo1} Lemma 3.3 that there are two $\Gamma$-orbits of roots in $\Lambda_{1,6}$. This determines a decomposition $\mathcal{H}=\mathcal{H}_n \cup \mathcal{H}_h$ where:
\begin{equation}\label{mirrortypes}
\begin{aligned}
\mathcal{H}_n &= \left\{ H_r \in \mathcal{H} \ ; \ H_r\cap \Lambda_{1,6} \cong \Lambda_2^2\oplus \left( \begin{smallmatrix}-2&0\\0&2\end{smallmatrix}\right) \right\} \\
\mathcal{H}_h &= \left\{ H_r \in \mathcal{H} \ ; \ H_r\cap \Lambda_{1,6} \cong \Lambda_2^2 \oplus \Lambda_{1,1} \right\}.
\end{aligned} 
\end{equation} 
A smooth point of a mirror $H_r \in \mathcal{H}_n$ corresponds to a plane quartic curve with a node and a smooth point of a mirror $H_r \in \mathcal{H}_h$ corresponds to a smooth hyperelliptic curve of genus three.

\subsection{The lattices $L_+$ and $L_-$}
The main result of this section is Lemma \ref{extend} which states that an antiunitary involution of the Gaussian lattice $\Lambda_{1,6}$ can be lifted to an involution of the $K3$ lattice such that its fixed point lattice is of hyperbolic signature. This will be an important ingredient in the proof of one of our main results: the real analogue of Kondo's period map for real quartic curves in Section \ref{Periodsofreal}. We start with a detailed analysis of the lattices $L_+$ and $L_-$. 

The lattice $L_+ \cong (2) \oplus A_1^7$ has an orthogonal basis $\{e_0,\ldots,e_7\}$ that satisfies $(e_0,e_0)=2$ and $(e_i,e_i)=-2$ for $i=1,\ldots,7$. According to Kondo the automorphism $\rho$ acts on $L_+$ by fixing the element $k=-3e_0+e_1+\ldots+e_7$ and acting as $-1$ on its orthogonal complement $k^\perp$ in $L_+$. This special element $k$ satisfies $(k,k)=4$ and represents the canonical class of the del Pezzo surface $Y$. The orthogonal complement $k^\perp$ is isomorphic to the root lattice $E_7(2)$. By the results of Section \ref{Vinberg} there is an isomorphism of groups:
\[O(L_+)\cong O(L_+)^+ \times \mathbb{Z}/2\mathbb{Z}\]
where the second factor is generated by $-1\in O(L_+)$. The group $O(L_+)^+$ is a hyperbolic Coxeter group as we have seen in Example \ref{VinCoxex} and its Coxeter diagram shown is Figure \ref{CoxeterLtau}. 
 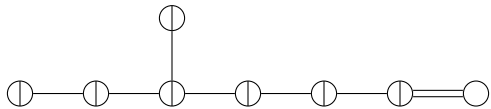
\begin{figure}[H]
\centering
\begin{tikzpicture}
\node[c,vert] (1) at (0,0) {};
\node[c,vert] (2) at (1,0) {};
\node[c,vert] (3) at (2,0) {};
\node[c,vert] (4) at (3,0) {};
\node[c,vert] (5) at (4,0) {};
\node[c,vert] (6) at (5,0) {};
\node[c,vert] (7) at (2,1) {};
\node[c] (8) at (6,0) {};
\draw[-] (1) -- (2) -- (3) -- (4) -- (5) -- (6);
\draw[-] (3) -- (7);
\draw[-,double distance = 2pt] (6) -- (8);
\end{tikzpicture}
\caption{The Coxeter diagram of the group $O(L_+)^+$.}
\label{CoxeterLtau}
\end{figure} 
From this diagram we see that the reflections in the long negative simple roots of $L_+$ form a subgroup $W(E_7)<O(L_+)^+$ of type $E_7$. It is precisely the stabilizer of the element $k \in L_+$. Recall from Section \ref{Lattices} that the discriminant group of a lattice $L$ is defined by $A_L=L^\vee/L$. Since the dual lattice $L_+^\vee$ can be naturally identified with the lattice $\frac{1}{2}L_+$ we have: 
\[A_{L_+} = \frac{1}{2}L_+/L_+ \cong (\mathbb{Z}/2\mathbb{Z})^8.\]  
\begin{prop}\label{mapL+}The natural map $O(L_+) \rightarrow O(A_{L_+})$ maps the subgroup $W(E_7)<O(L_+)^+$ isomorphically onto $O(A_{L_+})$.
\end{prop}
\begin{proof}
The bilinear form on $L_+\cong (2)\oplus A_1^7$ is even valued so that a reflection $s_r$ in a short root $r$ of norm $\pm 2$ satisfies:
\[ s_r(x) = x\pm (r,x)r \equiv x \bmod{L_+} \] 
for  $x\in \frac{1}{2}L_+$. This implies that these reflections are contained in the kernel of the map $O(L_+)\rightarrow O(A_{L_+})$. As a consequence the image of this map is generated by the subgroup $W(E_7)<O(L_+)^+$ of reflections in negative simple long roots. According to Kondo \cite{Kondo1} Lemma 2.2 the group $O(A_{L_+})$ is isomorphic to $W(E_7)^+ \times \mathbb{Z}/2\mathbb{Z} \cong W(E_7)$. Since the natural map $O(L_+)\rightarrow O(A_{L_+})$ is surjective, the proposition follows by Theorem \ref{Nik363}. \end{proof}

The $K3$ lattice $L$ is an even unimodular lattice and the primitive sublattices $L_+$ and $L_-$ satisfy: $L_-^\perp = L_+$. According to Proposition \ref{gluing} there is a natural isomorphism $O(A_{L_-}) \cong O(A_{L_+})$ which allows us to identify these groups. In particular we have $O(A_{L_-})\cong W(E_7)^+ \times \mathbb{Z}/2\mathbb{Z}$. We prefer to consider $L_-$ as the Gaussian lattice $\Lambda_{1,6}$ so that
\begin{equation}
\begin{aligned}
A_{\Lambda_{1,6}}&=\Lambda_{1,6}^\vee / \Lambda_{1,6} \\
& \cong \left( \frac{1}{1+i}\mathcal{G}/\mathcal{G} \right)^6 \times \frac{1}{2}\mathcal{G}/\mathcal{G}
\end{aligned}
\end{equation}
because $\Lambda_{1,6} = \Lambda_{1,5} \oplus \mathcal{G}(-2)$ with $\Lambda_{1,5}=\Lambda_2^2 \oplus \Lambda_{1,1}$ and $\Lambda_{1,5} = (1+i)\Lambda_{1,5}^\vee$.
\begin{rmk}\label{rhoaction}Note that there are isomorphism of additive groups $\frac{1}{1+i}\mathcal{G}/\mathcal{G} \cong \mathbb{Z}/2\mathbb{Z}$ and $\frac{1}{2}\mathcal{G}/\mathcal{G} \cong \mathbb{Z}/2\mathbb{Z} \times \mathbb{Z}/2\mathbb{Z}$. The generators of this last group are $\frac{1}{2}$ and $\frac{i}{2}$ and they are exchanged by multiplication by $i$.
\end{rmk}

\begin{prop}\label{compos}The composition of homomorphisms:
\[ U(\Lambda_{1,6}) \rightarrow O(A_{\Lambda_{1,6}}) \cong W(E_7)^+ \times \mathbb{Z}/2\mathbb{Z}\]
is given by reduction modulo $1+i$ on the first factor and the second factor is generated by the image of the central element $\rho \in U(\Lambda_{1,6})$.
\end{prop}
\begin{proof}
Let $A_{\Lambda_{1,6}}^\prime$ be the subset of $A_{\Lambda_{1,6}}$ where the discriminant quadratic form takes values in $\mathbb{Z}/2\mathbb{Z}$. The Gaussian lattice $\Lambda_{1,6}$ satisfies $\Lambda_{1,6}\subset (1+i)\Lambda_{1,6}^\vee$ so that the following equalities hold:
\begin{align*}
A_{\Lambda_{1,6}}^\prime 
&= \{ x\in \Lambda_{1,6}^\vee /\Lambda_{1,6} \ ; \ h(x,x)\in \mathbb{Z} \} \\
 &= \frac{1}{1+i}{\Lambda_{1,6}} / {\Lambda_{1,6}}.
\end{align*}
By writing: $h(\frac{1}{1+i}x,\frac{1}{1+i}x)=\frac{1}{2}h(x,x)$ for $x\in \Lambda_{1,6}$ we see that the $\mathbb{F}_2$-vectorspace $A_{\Lambda_{1,6}}^\prime$ with its induced quadratic form $q_{\Lambda_{1,6}}$ is isomorphic to the quadratic space $(V,q)$ from Proposition \ref{reduction}. According to this proposition there is an isomorphism $O(A_{\Lambda_{1,6}}^\prime) \cong W(E_7)^+$ and the composition of natural maps: 
\[U(\Lambda_{1,6})\rightarrow O(A_{L_-})\rightarrow O(A_{L_-}^\prime) \cong W(E_7)^+\]
corresponds to mapping an element $g\in U(\Lambda_{1,6})$ to its reduction $\overline{g}$ modulo  $(1+i)$. The automorphism $\rho \in U(\Lambda_{1,6})$ corresponds to multiplication by $i$ and by definition commutes with every element in $U(\Lambda_{1,6})$. It maps to the identity in $O(A_{\Lambda_{1,6}}^\prime)$ but acts as a nontrivial involution in $O(A_{\Lambda_{1,6}})$ by Remark \ref{rhoaction}. This implies that $O(A_{\Lambda_{1,6}})$ is isomorphic to the direct product of $O(A_{\Lambda_{1,6}}^\prime)$ with the subgroup $\mathbb{Z}/2\mathbb{Z} \triangleleft  O(A_{\Lambda_{1,6}})$ generated by $\rho$. \end{proof}

\begin{lem}\label{extend}Let $\chi_- \in U(\Lambda_{1,6})^a$ be an antiunitary involution of $\Lambda_{1,6}$. There is a unique $\chi \in O(L)$ that restricts to $\chi_-$ on $L_-$ so that the fixed point lattice $L^\chi$ is of hyperbolic signature.\end{lem}

\begin{proof}
Since complex conjugation on $\Lambda_{1,6}$ induces the identity on $O(A_{\Lambda_{1,6}})$ the statement of Proposition \ref{compos} is also true for the composition of homomorphisms:
\[ U(\Lambda_{1,6})^a \rightarrow O(A_{\Lambda_{1,6}})\cong W(E_7)^+\times \mathbb{Z}/2\mathbb{Z}.\]
Consider the image of the antiunitary involution $\chi_- \in U(\Lambda_{1,6})^a$ under this composition. This image is of the form $(\bar{u},\pm 1)$ where the involution $\bar{u}\in W(E_7)^+$ is obtained by reducing $\chi_- \in U(\Lambda_{1,6})^a$ modulo $(1+i)$. Observe that if the antiunitary involution $\chi_-$ maps to $(\bar{u},1)$ then $i\chi$ maps to $(\bar{u},-1)$. The involution 
\begin{equation} \chi_+=(\pm u,-1) \in W(E_7) \times \mathbb{Z}/2\mathbb{Z} < O(L_+) \end{equation}
maps to $(\bar{u},\pm1) \in O(A_{L_+})$ by Proposition \ref{mapL+}. Since $\chi_+$ maps $k \mapsto -k$ and $(k,k)=4$ the lattice of fixed points $L_+^{\chi_+}$ is negative definite. By Proposition \ref{gluing} there is a unique involution $\chi \in O(L)$ that restricts to $\chi_- \in U(\Lambda_{1,6})$ and $\chi_+ \in O(L_+)$ respectively. Since $\Lambda_{1,6}^{\chi_-}$ is of hyperbolic signature and $L_+^{\chi_+}$ is negative definite the fixed point lattice $L^\chi$ is of hyperbolic signature.
\end{proof}

\begin{prop}
Consider the $12$ antiunitary involutions $\chi_j$ and $i\chi_j$ for $j=1,\ldots,6$ from Theorem \ref{sixinvolutions}. For each of them the corresponding involution $\chi_+ \in O(L_+)$ is of the form $(u,-1)\in W(E_7)\times \mathbb{Z}/2\mathbb{Z}$. The conjugation classes of the involutions $u\in W(E_7)$ are shown in Table \ref{12conjclass}.
\begin{table}[H]
\centering
\[
\begin{array}{c|cccccc}
\toprule
j  & 1& 2& 3& 4& 5 & 6\\
\midrule
\chi_j & 1 & A_1 & A_1^2 & A_1^3 & D_4 & D_4\\
i\chi_j & E_7 & D_6 & D_4A_1 & A_1^4 & A_1^{3\prime} & A_1^{3\prime} \\
\bottomrule
\end{array}
\]
\caption{The conjugation classes in $W(E_7)$ of the $12$ antiunitary involutions $\chi_j,i\chi_j \in U(\Lambda_{1,6})^a$ for $j=1,\ldots,6$.}
\label{12conjclass}

\end{table}
\end{prop}
\begin{proof}
This follows from Table \ref{dimfix} and the proof of Lemma \ref{extend}.
\end{proof}

\subsection{Periods of real quartic curves}\label{Periodsofreal}

Let $C=\{ f(x,y,z) = 0 \} \subset \mathbb{P}^2$ be a smooth real plane quartic curve. This means that $C$ is invariant under complex conjugation of $\mathbb{P}^2(\mathbb{C})$ or equivalently that the polynomial $f$ has real coefficients. The $K3$ surface $X$ that corresponds to $C$ is also defined by an equation with real coefficients. Complex conjugation on $\mathbb{P}^3(\mathbb{C})$ induces an antiholomorphic involution $\chi_X$ on $X$.

\begin{dfn}
A $K3$ surface $X$ is called real if it is equipped with an antiholomorphic involution $\chi_X$. We will also call such an involution a real form of $X$. The real points of $X$, which we denote by $X(\mathbb{R})$, are the fixed points of the real form.  
\end{dfn}

\begin{thm}\label{realK3}Let $\chi$ be an involution on the $K3$ lattice $L$. There exists a marked $K3$ surface $(X,\phi)$ such that $\chi_X=\phi^{-1}\circ \chi \circ \phi$ induces a real form on $X$ if and only if the lattice of fixed points $L^\chi$ has hyperbolic signature. 
\end{thm}
\begin{proof}
See \cite{Silhol} Chapter VIII Theorem $2.3$.
\end{proof} 

Suppose $(X,\chi_X)$ is a real $K3$ surface. By choosing a marking we obtain an involution $\chi$ of the $K3$ lattice $L$. By Theorem \ref{realK3} the fixed point lattice of this involution $L^\chi$ is of hyperbolic signature. Since the $K3$ lattice is an even unimodular lattice, the lattice $L^\chi$ is even and $2$-elementary. According to Proposition \ref{2invariants} the isomorphism type of $L^\chi$ is determined by three invariants $(r,a,\delta)$ where $r=r_++r_-=1+r_-$. It is clear that these invariants do not depend on the marking of $X$. The following theorem originally due to Kharlamov \cite{KharTop} shows that they determine the topological type of the real point set $X(\mathbb{R})$. We will write $S_g$ for a real orientable surface of genus $g$ and $kS$ for the disjoint union of $k$ copies of a real surface $S$.

\begin{thm}[Nikulin \cite{NikulinSym} Thm. 3.10.6]\label{NikulinRealInv}
Let $(X,\chi_X)$ be a real $K3$ surface. Then:
\[ 
X(\mathbb{R}) = \begin{cases}\emptyset & \text{if } (r,a,\delta) = (10,10,0) \\ 2S_1 & \text{if } (r,a,\delta) = (10,8,0)  \\
S_g \sqcup kS_0 & \text{otherwise} \end{cases}
\]
where $g=\frac{1}{2}(22-r-a)$ and $k=\frac{1}{2}(r-a)$. 
\end{thm}

\begin{rmk}There are two antiholomorphic involutions on the $K3$ surface $X = \{w^4 = f(x,y,z)\}$. Since we chose the sign of $f(x,y,z)$ to be positive on the interior of the curve $C(\mathbb{R})$ the  antiholomorphic involution $\chi_X$ is determined without ambiguity.
\end{rmk}
By fixing a marking $\phi:H^2(X,\mathbb{Z}) \rightarrow L$ of the $K3$ surface $X$ we associate to $\chi_X$ the involution:
\[ \chi=\phi \circ \chi_X^\ast \circ \phi^{-1}\] 
of the $K3$ lattice $L$. Since the involution $\chi$ commutes with $\tau$ it preserves the $\pm 1$-eigenlattices of the involution $\tau$. We denote by $\chi_-$ (resp. $\chi_+$) the induced involution on $L_-$ (resp. $L_+$). It is clear that $\chi$ and $\rho$ satisfy the relation:
\[ \rho \circ \chi = \tau \circ \chi \circ \rho \]
so that on the eigenlattice $L_-$ where $\tau$ acts as $-1$ they anticommute and on $L_+$ they commute. This implies that $\chi_-$ is an antiunitary involution of the Gaussian lattice $\Lambda_{1,6}$.

By the results of Section \ref{Kondomain} on Kondo's period map we can associate to a smooth real plane quartic curve $C$ a period point $[x]\in \mathbb{B}^\circ$ and the real form $[\chi_-]$ of $\mathbb{B}$ we just defined fixes $[x]$. The following lemma shows that the $P\Gamma$-conjugation class of $[\chi_-]$ does not change if we vary $C$ in its connected component of $\mathcal{Q}^\mathbb{R}$. 
\begin{lem}\label{realisom}
If two smooth real plane quartic curves $C$ and $C^\prime$ are real isomorphic then the projective classes $[\chi_-]$ and $[\chi_-^\prime]$ of their corresponding antiunitary involutions in $\Lambda_{1,6}$ are conjugate in $P\Gamma$.
\end{lem}
\begin{proof}
Since $C$ and $C^\prime$ are real plane curves a real isomorphism $C\rightarrow C^\prime$ is induced from an element in $PGL(3,\mathbb{R})$. We can lift this element to $PGL(4,\mathbb{R})$ so that it induces an isomorphism $\alpha_C:X\rightarrow X^\prime$ that commutes with the covering transformations $\rho_X$ and $\rho_{X^\prime}$ of $X$ and $X^\prime$. Since the real forms $\chi_X$ and $\chi_X^\prime$ of $X$ and $X^\prime$ are both induced by complex conjugation on $\mathbb{P}^3$ they satisfy $\chi_X^\prime = \alpha_C \circ \chi_X \circ \alpha_C^{-1}$. By fixing markings of the $K3$ surfaces $X$ and $X^\prime$ we obtain induced orthogonal transformations $\chi,\chi^\prime$ and $\alpha$ of the $K3$-lattice $L$ such that $\chi^\prime = \alpha \circ \chi \circ \alpha^{-1}$. Since $\alpha$ commutes with $\rho$ the restriction $\alpha_-$ of $\alpha$ to $L_-$ is contained in $\Gamma$. This proves the lemma.\end{proof}

Let $\mathbb{B}^{\chi_-}$ be the fixed point set in $\mathbb{B}$ of the real form $[\chi_-]$. The fixed point lattice $\Lambda_{1,6}^{\chi_-}$ has hyperbolic signature $(1,6)$ so that $\mathbb{B}^{\chi_-}$ is the real hyperbolic ball
 \[ \mathbb{B}^{\chi_-} = \mathbb{P}\{ x\in \Lambda_{1,6}^{\chi_-} \otimes_\mathbb{Z} \mathbb{R} \ ; \ h(x,x) >0 \} .\]
As before we denote by $P\Gamma^{\chi_-}$ the stabilizer of $\mathbb{B}^{\chi_-}$ in the ball $\mathbb{B}$. Since the period point of a smooth real quartic curve $C$ is fixed by $[\chi_-]$ it lands in the real ball quotient: $P\Gamma^{\chi_-}\backslash (\mathbb{B}^{\chi_-})^\circ$. This gives rise to a real period map. More precisely we have the following real analogue of Theorem \ref{Kondomain}.

\begin{thm}\label{realK3period}The real period map $\Per^\mathbb{R}$ that maps a smooth real plane quartic to its period point in $P\Gamma \backslash \mathbb{B}^\circ$ defines an isomorphism of real analytic orbifolds:
\begin{equation} \Per^\mathbb{R}: \mathcal{Q}^\mathbb{R} \rightarrow \coprod_{[\chi_-]} P \Gamma^{\chi_-} \big\backslash  \left( \mathbb{B}^{\chi_-} \right)^\circ  
\end{equation}
where $[\chi_-]$ varies over the $P\Gamma$-conjugacy classes of projective classes of antiunitary involutions of $\Lambda_{1,6}$. 
\end{thm}

\begin{proof}
We construct an inverse to the real period map. Let $z \in \mathbb{B}^\circ$ be such that $\chi_-(z)=z$ for a certain antiunitary involution of $\Lambda_{1,6}$. From the proof of Theorem \ref{Kondomain} we see that there is a marked $K3$ surface $X$ that corresponds to $z$. According to Lemma \ref{extend} the involution $\chi_-$ lifts to an involution $\chi \in O(L)$ such that for its restriction $\chi_+$ to $L_+$ the fixed point lattice $L_+^{\chi_+}$ is negative definite. Since $\Lambda_{1,6}^{\chi_-}$ is of hyperbolic signature the lattice $L^\chi$ is also of hyperbolic signature. According to Theorem \ref{realK3} this implies that the marked $K3$ surface $X$ is real. Its real form $\chi_X$ commutes with $\tau_X$ so that it induces a real form on $\chi_Y$ on the del Pezzo surface $Y=X/\left< \tau_X \right>$. The anticanonical system $|-K_Y|:Y\rightarrow \mathbb{P}^2$ is the double cover of $\mathbb{P}^2$ ramified over a smooth real plane quartic curve $C$. The inverse of the real period map associates to the $P\Gamma^{\chi_-}$ orbit of $z\in (\mathbb{B}^{\chi_-})^\circ$ the real isomorphism class of the real quartic curve $C$.   \end{proof}

\subsection{The six components of $\mathcal{Q}^\mathbb{R}$}
\label{sixcomponents}
In this section we complete our description of the real period map $\Per^\mathbb{R}$ by connecting the six connected components of the moduli space $\mathcal{Q}^\mathbb{R}$ of smooth real plane quartic curves to the six projective classes of antiunitary involutions of the Gaussian lattice $\Lambda_{1,6}$ from Theorem \ref{sixinvolutions}. We first prove that these six antiunitary involutions are in fact all of them.

\begin{prop}\label{allsixinvolutions}There are six projective classes of antiunitary involutions of the Gaussian lattice $\Lambda_{1,6}$ up to conjugation by $P\Gamma$. \end{prop}
\begin{proof}
Since $\mathcal{Q}^\mathbb{R}$ consists of six connected components and the real period map $\Per^\mathbb{R}$ is surjective the number of projective classes is at most six. In Theorem \ref{sixinvolutions} we found six projective classes of antiunitary involutions up to conjugation by $P\Gamma$ so these six are all of them.

\end{proof}

The following corollary follows from the proof of Theorem \ref{realK3period}.
\begin{cor}Suppose $z\in \mathbb{B}^\circ$ is a real period point so that it is fixed by an antiunitary involution $\chi \in U(\Lambda_{1,6})^a$. By the real period map we associate to $z \in \mathbb{B}^\circ$ a real del Pezzo surface $Y$ of degree two together with a marking
\[ H^2(Y,\mathbb{Z})\rightarrow L_+\left(\tfrac{1}{2}\right) \]
such that the induced involution of the real form of $Y$ on $L_+(\frac{1}{2})$ is given by $\chi_+$.
\end{cor}
 
We review some results of \cite{Wall} on real del Pezzo surfaces of degree two. Other references on this subject are Koll\'ar \cite{Kollar} and Russo \cite{Russo}. A real del Pezzo surface $Y$ of degree two is the double cover of the projective plane $\mathbb{P}^2$ ramified over a smooth real plane quartic curve $C\subset \mathbb{P}^2$ so that:
\[ Y = \{ w^2 = f(x,y,z) \}. \]
We choose the sign of $f$ so that $f>0$ on the orientable interior part of $C(\mathbb{R})$. By using the deck transformation $\rho_Y$ of the cover we see that
\begin{equation}
\begin{aligned}
\chi_Y^+: \left[ w:x:y:z \right] & \mapsto [\bar{w}:\bar{x}:\bar{y}:\bar{z}] \\
\chi_Y^-: [w:x:y:z] & \mapsto [-\bar{w}:\bar{x}:\bar{y}:\bar{z}] \\
\end{aligned}
\end{equation}\label{realstructuredelpezzo}
are the two real forms of $Y$. These real forms satisfy $\chi_Y^-=\rho_Y \circ \chi_Y^+$ and we denote the real point sets of $\chi_Y^+$ and $\chi_Y^-$ by $Y^+(\mathbb{R})$ and $Y^-(\mathbb{R})$ respectively. Note that $Y^+(\mathbb{R})$ is an orientable surface while $Y^-(\mathbb{R})$ is nonorientable. Suppose that $H^2(Y,\mathbb{Z})\rightarrow L_+(\frac{1}{2})$ is a marking of $Y$. The deck transformation $\rho_Y$ induces the involution:
\[ \rho = (-1,1) \in W(E_7)\times \mathbb{Z}/2\mathbb{Z}. \]
in $O(L_+(\frac{1}{2}))$. This implies that the two real forms $\chi_Y^\pm$ form a pair
\[ (\chi_Y^+,\chi_Y^-) \longleftrightarrow (\pm u , -1) \in W(E_7)\times \mathbb{Z}/2\mathbb{Z}. \]
In \cite{Wall} Wall determines the correspondence between the conjugation classes of the $u\in W(E_7)$ and the topological type of $Y(\mathbb{R})$. The results are shown in Table \ref{realdelpezzotable}. We use the notation $kX$ for the disjoint union and $\# kX$ for the connected sum of $k$ copies of a real surface $X$. From this table we see that except for the classes of $D_4$ and $A_1^{3\prime}$ the conjugation class of $u\in W(E_7)$ determines the topological type of the real plane quartic curve $C(\mathbb{R})$.

\begin{table}

\[
\begin{array}{ccll}
\toprule
j& C(\mathbb{R}) & u\in W(E_7) & Y(\mathbb{R}) \\
\midrule
\multirow{2}{*}{1} & \multirow{2}{*}{\begin{tikzpicture}\draw (45:10pt) circle [radius=5pt];\draw (135:10pt) circle [radius=5pt];\draw (-135:10pt) circle [radius=5pt];\draw (-45:10pt) circle [radius=5pt];\end{tikzpicture}} & 1 & \# 8\mathbb{P}^2(\mathbb{R}) \\
& & E_7 & 4S_0\\
\\
\multirow{2}{*}{2} & \multirow{2}{*}{\begin{tikzpicture}\draw (0:8pt) circle [radius=5pt];\draw (120:8pt) circle [radius=5pt];\draw (240:8pt) circle [radius=5pt];\end{tikzpicture}} & A_1 & \# 6\mathbb{P}^2(\mathbb{R}) \\
& & D_6 & 3S_0\\
\\
\multirow{2}{*}{3} & \multirow{2}{*}{\begin{tikzpicture}\draw (0:7pt) circle [radius=5pt]; \draw (180:7pt) circle [radius=5pt];\end{tikzpicture}} & A_1^2 & \# 4\mathbb{P}^2(\mathbb{R}) \\
& & D_4A_1 & 2S_0\\
\\
\multirow{2}{*}{4} & \multirow{2}{*}{\begin{tikzpicture}\draw (0:0) circle [radius=5pt];\end{tikzpicture}} & A_1^3 & \# 2\mathbb{P}^2(\mathbb{R}) \\
& & A_1^4 & S_0\\
\\
\multirow{2}{*}{5} & \multirow{2}{*}{ \begin{tikzpicture}[baseline = -3pt]\draw (0,0) circle [radius=8pt];\draw (0,0) circle [radius=4pt];\end{tikzpicture}} & D_4 & S_0 \sqcup \# 2\mathbb{P}^2(\mathbb{R}) \\
& & A_1^{3\prime} & S_1 \\
\\
 \multirow{2}{*}{6} & \multirow{2}{*}{ $\emptyset$ } & D_4 &  2\mathbb{P}^2(\mathbb{R}) \\
&  & A_1^{3\prime} & \emptyset \\
\bottomrule
\end{array}
\]
\caption{The real topological types of real del Pezzo surfaces of degree two and their corresponding involutions in the Weyl group $W(E_7)$.}
\label{realdelpezzotable}

\end{table}

\begin{thm}
The correspondence between the six projective classes of antiunitary involutions of the lattice $\Lambda_{1,6}$ up to conjugation by $P\Gamma$ and the real components of $\mathcal{Q}^\mathbb{R}$ is given by
\[ \mathcal{Q}_j^\mathbb{R} \longleftrightarrow \chi_j  \qquad j=1,\ldots,6. \]
The index $j$ on the left is given by Table \ref{realdelpezzotable} and the index $j$ on the right by Table \ref{12conjclass}.
\end{thm}
\begin{proof}
For $j=1,2,3,4$ the statement follows by comparing Table \ref{realdelpezzotable} and Table \ref{12conjclass}. Unfortunately this does not work for the projective classes antiunitary involutions $\chi_5$ and $\chi_6$ since both correspond to the involutions $D_4\in W(E_7)$. To distinguish these two we will prove that the antiunitary involution $i\chi_6$ extends to an involution of the $K3$ lattice whose real $K3$ surface $X$ has no real points. This implies that the projective class of $\chi_6$ corresponding to the component $\mathcal{Q}^\mathbb{R}_6$ of smooth real quartic curves with no real points. 

For this let $L=U^3 \oplus E_8^2$ be the $K3$ lattice and consider the involution:
\begin{equation}\label{emptyinv}
\chi = -I_2 \oplus \begin{pmatrix}0&I_2\\I_2&0\end{pmatrix} \oplus \begin{pmatrix}0&I_8\\I_8&0\end{pmatrix} \in O(L).
\end{equation}
It is clear from the expression for $\chi$ that the fixed point lattice $L^\chi$ is isomorphic to $U(2)\oplus E_8(2)$. The invariants $(r,a,\delta)$ of this lattice are given by $(10,10,0)$ so that $X(\mathbb{R}) = \emptyset$ according to Theorem \ref{NikulinRealInv}. Using the explicit embedding of $L_+$ and $L_-$ into the $K3$ lattice $L$ from Lemma \ref{explicitembedding} it is easily seen that: 
\begin{equation}L_-^{\chi_-} \cong U(2)\oplus D(4) \oplus A_1(2) \quad , \quad L_+^{\chi_+} \cong A_1(2)^3. \end{equation}
By consulting Table \ref{12lattices} we now deduce that $\chi$ is conjugate to $i\chi_6$ in $P\Gamma$.

\end{proof}

\subsection{The geometry of maximal quartics}

We now study the component $\mathcal{Q}_1^\mathbb{R} \cong P\Gamma^{\chi_1} \backslash \mathbb{B}_6^{\chi_1}$ that corresponds to $M$-quartics in more detail. An $M$-quartic is a smooth real plane quartic curve $C$ such that its set of real points $C(\mathbb{R})$ consists of four ovals. Much of the geometry of such quartics is encoded by a hyperbolic polytope $C_6 \subset \mathbb{B}_6^{\chi_1}$. 
\begin{thm}The group $P\Gamma^{\chi_1}$ is isomorphic to the semidirect product
\[ W(C_6) \rtimes \Aut(C_6) \]
where $C_6\subset \mathbb{B}_6^{\chi_1}$ is the hyperbolic Coxeter polytope whose Coxeter diagram is shown in Figure \ref{Coxeter6}. Its automorphism group $\Aut(C_6)$ is isomorphic to the symmetric group $S_4$.
\end{thm}
\begin{proof}
Recall that an element $ [g]\in P\Gamma^\chi$ is of type $II$ if and only if there is a $g\in [g]$ such that $g\Lambda_{1,6}^\chi = \Lambda_{1,6}^{i\chi}$. We see from Table \ref{12lattices} that the lattice $\Lambda_{1,6}^{\chi}$ is not isomorphic to the lattice $\Lambda_{1,6}^{i\chi}$ so that the group $P\Gamma^\chi$ does not contain elements of type $II$. Therefore the group $P\Gamma^{\chi_1}$ consists of all element of $PO(\Lambda^{
\chi_1}_{6,1})$ that are induced from $U(\Lambda_{1,6})$. The lattice $\Lambda_{1,6}^{\chi_1}$ is isomorphic to $(2) \oplus A_1^6$. A basis $\{e_0,\ldots,e_6\}$ in $\Lambda_{1,6}$ is given by the columns of the matrix:
\[
B_1 = \begin{pmatrix}0&1+i&0&0&0&0&0\\0&1&1&0&0&0&0\\0&0&0&1+i&0&0&0\\0&0&0&1&1&0&0\\0&0&0&0&0&1+i&0\\0&0&0&0&0&1&1\\1&0&0&0&0&0&0\end{pmatrix}.
\] 
It is a reflective lattice and the group $PO(\Lambda_{1,6}^{\chi_1})$ is a Coxeter group whose Coxeter diagram can be found in Figure \ref{E7diagrams}. A reflection $s_{r}\in PO(\Lambda_{1,6}^{\chi_1})$ is induced from $U(\Lambda_{1,6})$ if and only if the root $r$ satisfies Equation \ref{realrootcondition}. Note that a vector $r=(z_1,\ldots,z_7)\in \Lambda_{1,6}\otimes_\mathcal{G} \mathbb{Q}$ is contained in $\Lambda_{1,6}^\vee$ if and only if $z_i \in \frac{1}{1+i}\mathcal{G}$ for $i=1,\ldots,6$ and $z_7 \in \frac{1}{2}\mathcal{G}$ so that we can rewrite this equation as:
\begin{align*}
  \frac{2(1+i)z_i}{h(r,r)}\in \mathcal{G} \quad \text{for } i=1,\ldots,6 \quad , \quad \frac{4z_7}{h(r,r)}\in \mathcal{G}.
\end{align*}
These equations are automatically satisfied if $h(r,r)=-2$ and if $h(r,r)=-4$ they are equivalent to: $(1+i)$ divides $z_i$ for $i=1,\ldots,6$. This can be checked from the matrix $B_1$. Now we run Vinberg's algorithm with this condition and the result is the hyperbolic Coxeter polytope $C_6$ shown in Figure \ref{Coxeter6}. The vertices $r_1,r_3,r_5$ and $r_{13}$ of norm $-4$ roots form a tetrahedron. Every symmetry of this tetrahedron extends to the whole Coxeter diagram. Consequently the symmetry group of the Coxeter diagram is the symmetry group of a tetrahedron which is isomorphic to $S_4$. Consider the two elements $s,t\in PO(\Lambda_{1,6}^{\chi_1})$ defined by
\begin{equation}\label{S4gens}
\begin{aligned}
s&=s_{e_4-e_6}\cdot s_{e_3-e_5} \cdot s_{e_1-e_3} \cdot s_{e_2-e_4} \\
t&=s_{e_0-e_1-e_3-e_4} \cdot s_{e_0-e_1-e_5-e_6}.
\end{aligned}
\end{equation}
The element $s$ has order three and corresponds to the rotation of the tetrahedron that fixes $r_{13}$ and cyclically permutes $(r_1r_5r_3)$. The element $t$ has order two and corresponds to the reflection of the tetrahedron that interchanges $r_1$ and $r_{13}$ and fixes $r_3$ and $r_5$. Together these transformations generate $S_4$. We can check that both are contained in $P\Gamma^{\chi_1}$ by using Equation \ref{realcondition}. 
\end{proof}
\begin{figure}
\centering

$\begin{array}{cccccccc}
\toprule
 & e_0 & e_1 & e_2 & e_3 & e_4 & e_5 & e_6 \\
 \midrule
 p & 1 & 0 & 0 & 0 & 0 & 0 & 0\\
 \text{height $0$} &&&&&&&\\
 r_1 & 0& 1& -1 & 0&0&0&0\\
 r_2 & 0&0&1&0&0&0&0\\
 r_3 & 0&0&0&1&-1&0&0\\
 r_4 & 0&0&0&0&1&0&0\\
 r_5 & 0&0&0&0&0&1&-1\\
 r_6 & 0&0&0&0&0&0&1\\
\text{height $2$} &&&&&&&\\
 r_7 & 1&-1&0&-1&0&0&0\\
 r_8 & 1&-1&0&0&0&-1&0\\
 r_9 & 1&0&0&-1&0&-1&0\\
 r_{10} & 1&-1&-1&0&0&0&0\\
 r_{11} & 1&0&0&-1&-1&0&0\\
 r_{12} & 1&0&0&0&0&-1&-1\\
\text{height $4$} &&&&&&&\\
r_{13} & 2&-1&-1&-1&-1&-1&-1\\
\bottomrule
\end{array}$

\begin{tikzpicture}[scale=1.2]
    \node[c,vert,label=left:$r_5$] (1) at ( 0,0) {};
    \node[c,label=above:$r_8$]             (2) at ( 6,0) {};    
    \node[c,vert,label=right:$r_1$] (3) at ( 8,0) {};
    \node[c,label=below right:$r_7$]            (4) at ( 5.5,-1.25) {};
    \node[c,vert,label=below:$r_3$] (5) at ( 4,-2) {};
    \node[c,label=below left:$r_9$]             (6) at ( 2,-1) {};    
    \node[c,vert,label=above:$r_{13}$] (7) at ( 4,4) {};
    \node[c,label=above left:$r_4$]             (8) at ( 4,2) {};    
    \node[c,label=above right:$r_2$]             (9) at ( 6.25,1.75) {};
    \node[c,label=above left:$r_6$]             (10) at (1.5,1.5) {};
    
    \node[c,fill=LightGray,label=above right:$r_{11}$] (x1) at (4.5,1.5) {};
    \node[c,fill=LightGray,label=below right:$r_{10}$] (x2) at (4.5,0.5) {};
    \node[c,fill=LightGray,label=below left:$r_{12}$] (x3) at (3,0.5) {};
    
    \draw [-,double distance=2pt] (1) -- (2) -- (3);
    
    \draw [-,ultra thick] (x1) -- (x2) -- (x3) -- (x1);
    \draw [-,ultra thick] (2) -- (x1) -- (8);
    \draw [-,ultra thick] (4)  --(x3) -- (10);
    \draw [-,draw=white,line width=4pt] (6) -- (x2) -- (9);
    
    \draw [-,ultra thick] (6) -- (x2) -- (9);

    \draw [-,draw=white,line width=5pt] (5) -- (8) -- (7);
    \draw [-,double distance=2pt] (5) -- (8) -- (7)  ;
    
    \draw [-,double distance=2pt] (5) -- (4) -- (3) ;
    \draw [-,double distance=2pt] (5) -- (6) -- (1);
    \draw [-,double distance=2pt] (1) -- (10) -- (7);
    \draw [-,double distance=2pt] (3) -- (9) -- (7);
    
\end{tikzpicture}
\caption{The Coxeter diagram of the reflection part of the group $P\Gamma^1$}
\label{Coxeter6}
\end{figure}
We see from the Coxeter diagram of the polytope $C_6$ that there are three orbits of roots under the automorphism group $\Aut(C_6)\cong S_4$. The orbit of a root $r$ corresponding to a grey node of norm $-2$ satisfies $r^\perp \cong \Lambda_2^2 \oplus \Lambda_{1,1}$. According to Equation \ref{mirrortypes} the mirror of such a root is of hyperelliptic type. This means that the smooth points of such a mirror correspond to a smooth hyperelliptic genus three curves. The Coxeter diagram of the wall that corresponds to the hyperelliptic root $r_{11}$ is the subdiagram consisting of the nodes belonging to the roots
\[ \{r_1,r_2,r_3,r_5,r_6,r_7,r_9,r_{13}\}. \]
It is isomorphic to the Coxeter diagram on the right hand side of Figure \ref{CoxDiag}. This is also the case for the other two hyperelliptic roots so they correspond to the maximal real component of  real hyperelliptic genus three curves.  

The other two orbits of roots satisfy $r^\perp \cong \Lambda_2^2 \oplus \left( \begin{smallmatrix}-2&0\\0&2\end{smallmatrix} \right)$ so that their mirrors are of nodal type. For a white root of norm $-2$ the orthogonal complement $r^\perp$ in the lattice $\Lambda_{1,6}^{\chi_1}$ is isomorphic to $(2)\oplus A_1^5$. The smooth points of such a mirror correspond to quartic curves with a nodal singularity such that the tangents at the node are real. Locally such a node is described by the equation $x^2-y^2=0$. This happens when two ovals touch each other. Since there are four ovals this can happen in $\binom{4}{2}=6$ ways; hence there are six mirrors of this type.

For a nodal root of norm $-4$ the orthogonal complement is given by $r^\perp \cong (2)\oplus A_1^4 \oplus A_1(2)$ in $\Lambda_{1,6}^{\chi_1}$. The smooth points of such a mirror correspond to quartic curves with a nodal singularity such the tangents at the node are complex conjugate. Locally this is described by $x^2+y^2=0$. It happens when an oval shrinks to a point which can occur for each of the four ovals, and so there are four mirrors of this type. 


A point $[x] \in C_6$ that is invariant under the action of $\Aut(C_6)\cong S_4$ corresponds to an $M$-quartic whose automorphism group is isomorphic to $S_4$. These points are described by the following lemma.
\begin{lem}
A point $[x]\in C_6$  with $x=(x_0,\ldots,x_6) \in \Lambda_{1,6}^{\chi_1}\otimes_{\mathbb{Z}}\mathbb{R}$ is invariant under $\Aut(C_6)\cong S_4$ if and only if it lies on the hyperbolic line segment
\[ L= \{(-2b-a,b,a,b,a,b,a)  \ ; \ a,b,\in \mathbb{R}, \ b \leq a \leq 0 \}/\mathbb{R}_+ \subset C_6. \]
The line segment $L$ has fixed distances $d_1,d_2$ and $d_3$ to mirrors of type \begin{tikzpicture}\node[c] (1) at (0,0) {};\end{tikzpicture},\begin{tikzpicture}\node[c,fill=LightGray] (1) at (0,0) {};\end{tikzpicture} and \begin{tikzpicture}\node[c,vert] (1) at (0,0) {};\end{tikzpicture} respectively, and these distances satisfy 
\[ [\sinh^2 d_1:\sinh^2 d_2:\sinh^2 d_3]=[a^2:b^2:(a-b)^2/2].\]
\end{lem}

\begin{proof}The group $\Aut(C_6)$ is generated by the two elements $s$ and $t$ from Equation \ref{S4gens}. A small computation shows that a point $x \in \Lambda_{1,6}^{\chi_1} \otimes_\mathbb{Z} \mathbb{R}$ is invariant under these two generators if and only if it is of the form 
\[ x = (-2b-a,b,a,b,a,b,a).\]
The second statement of the Lemma follows from the formula for hyperbolic distance (Equation \ref{hypdistance}) and the equalities
\begin{align}
(x,r_i) = \begin{cases} -a & i=2,4,6,7,8,9 \\
-b & i=10,11,12 \\
a-b & i=1,3,5,13 .
\end{cases}
\end{align}
\end{proof}  
The line segment $L$ connects the vertex $L_0=(-2,1,0,1,0,1,0)\in C_6$ of type $A_1^6$  to the point $L_1=(-3,1,1,1,1,1,1)$. A consequence of the real period map of Theorem \ref{realK3period} is that there is a unique one-parameter family of smooth plane quartics with automorphism group $S_4$ that corresponds to the line segment $L\subset C_6$. This pencil was previously studied by W.L Edge \cite{Edge}. It is described by the following proposition.

\begin{prop}The one-parameter family of quartic curves $C_t$ by:
\[ C_t = \prod(\pm x \pm y +z) + t(x^4+y^4+z^4) \quad, \quad 0\leq t\leq 1. \]
corresponds to the line segment $L$ under the real period map.
\end{prop} 
\begin{proof}This family is invariant under permutations of the coordinates $(x,y,z)$ and the transformations: $(x,y,z)\mapsto (\pm x,\pm y,z)$. Together these generate a group $S_3\rtimes V_4 \cong S_4$. The curve $C_0$ is a degenerate quartic that consists of four lines and has six real nodes corresponding to the intersection points of the lines. For $0<t<1$ the curve $C_t$ is an $M$-quartic. The quartic $C_1$ has no real points except for four isolated nodes.  
\end{proof}

\begin{rmk}
A plane quartic curve with a single node has $22$ bitangents, because the $6$ bitangents through the node all have multiplicity $2$. If a real plane quartic has $4$ smooth ovals in the real projective plane, then this curve has $24$ real bitangents intersecting the quartic in $2$ real points on $2$ distinct ovals (indeed, one has $4$ such bitangents for each pair of ovals). Each nonconvex oval gives a real bitangent intersecting that oval in $2$ real points. Each convex oval gives a real bitangent intersecting the quartic in $2$ complex conjugate points. The conclusion is that a real plane quartic with $4$ smooth ovals in the real projective plane has $28$ bitangents, and therefore this plane quartic curve can have no complex singular points. In turn this implies that the moduli space of maximal real quartics is a contractible orbifold, and is in fact an open convex polytope modulo an action of $S_4$. 
The same phenomenon holds for the moduli space of maximal real octics, which is again a contractible orbifold. However nonmaximal real octics with a smooth real locus might have complex singular points, and a similar phenomenon is to be expected for nonmaximal real quartic curves. 
\end{rmk}

\begin{figure}
\centering
\subfigure[$t=0$]{
\includegraphics[width=5cm,trim= 6 6 6 6,clip]{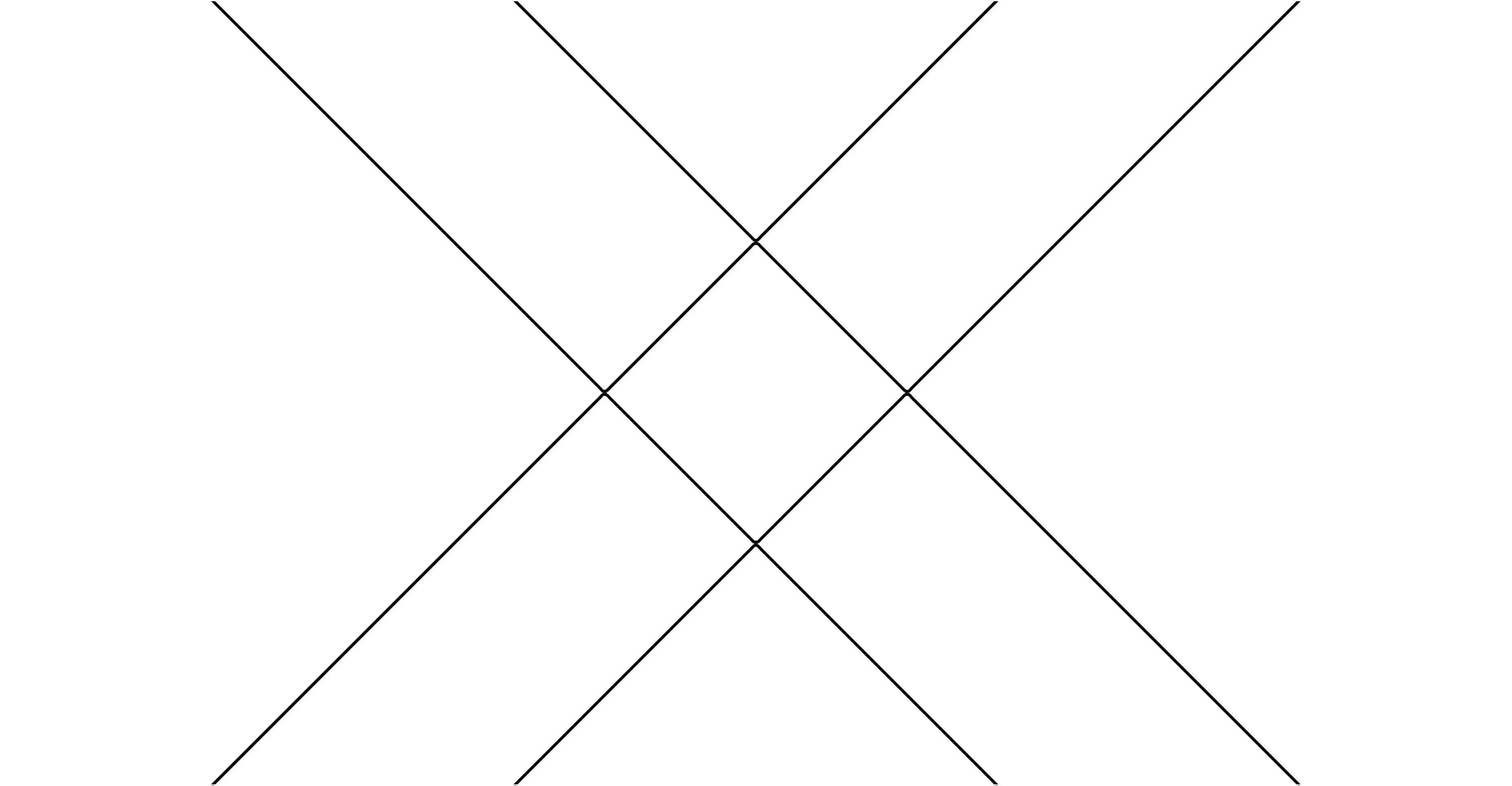} 
}
\subfigure[$t=\frac{1}{2}$]{
\includegraphics[width=5cm,trim= 6 6 6 6,clip]{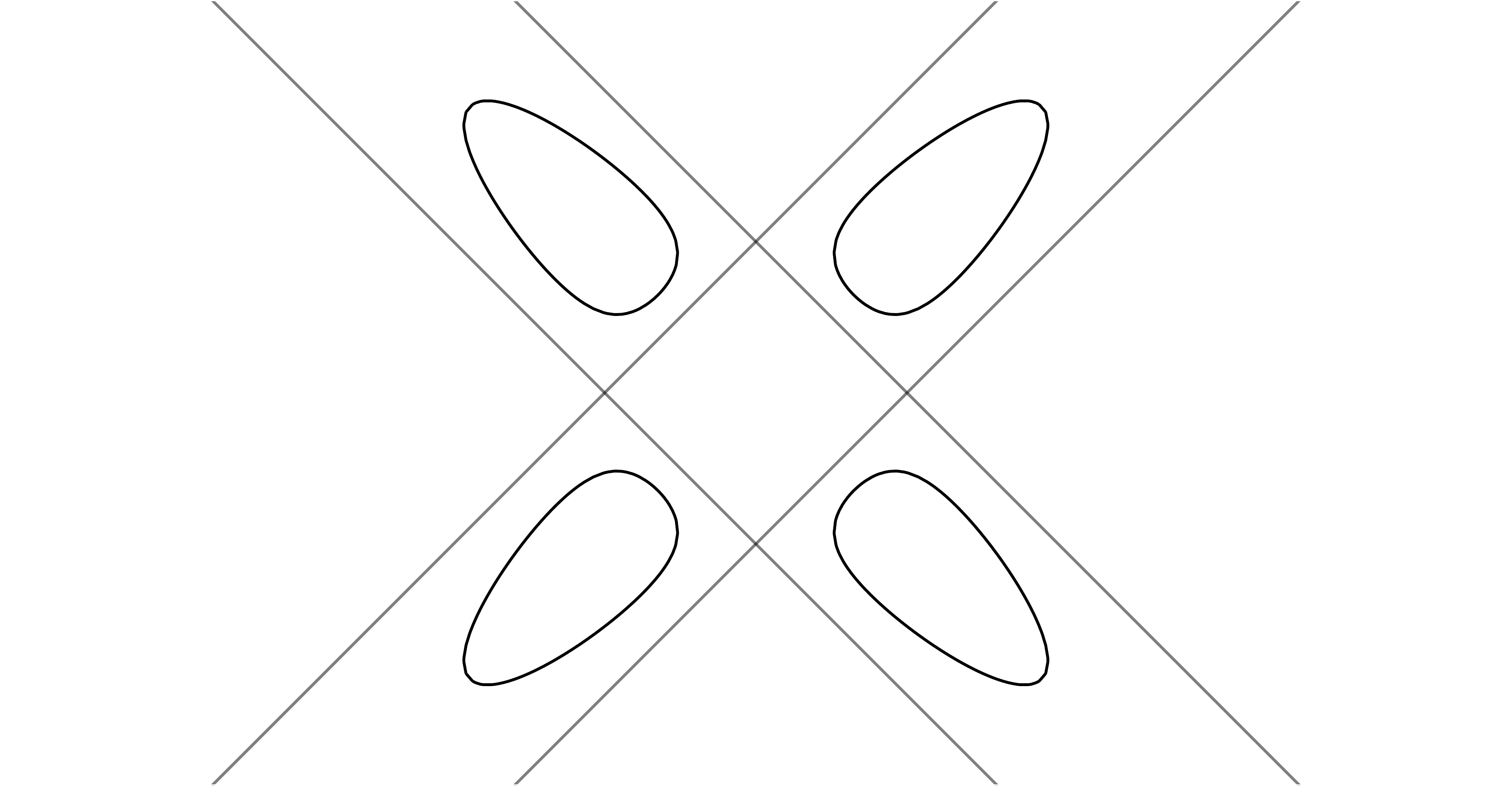} 
}
\subfigure[$t=1$]{
 \includegraphics[width=5cm,trim= 6 6 6 6,clip]{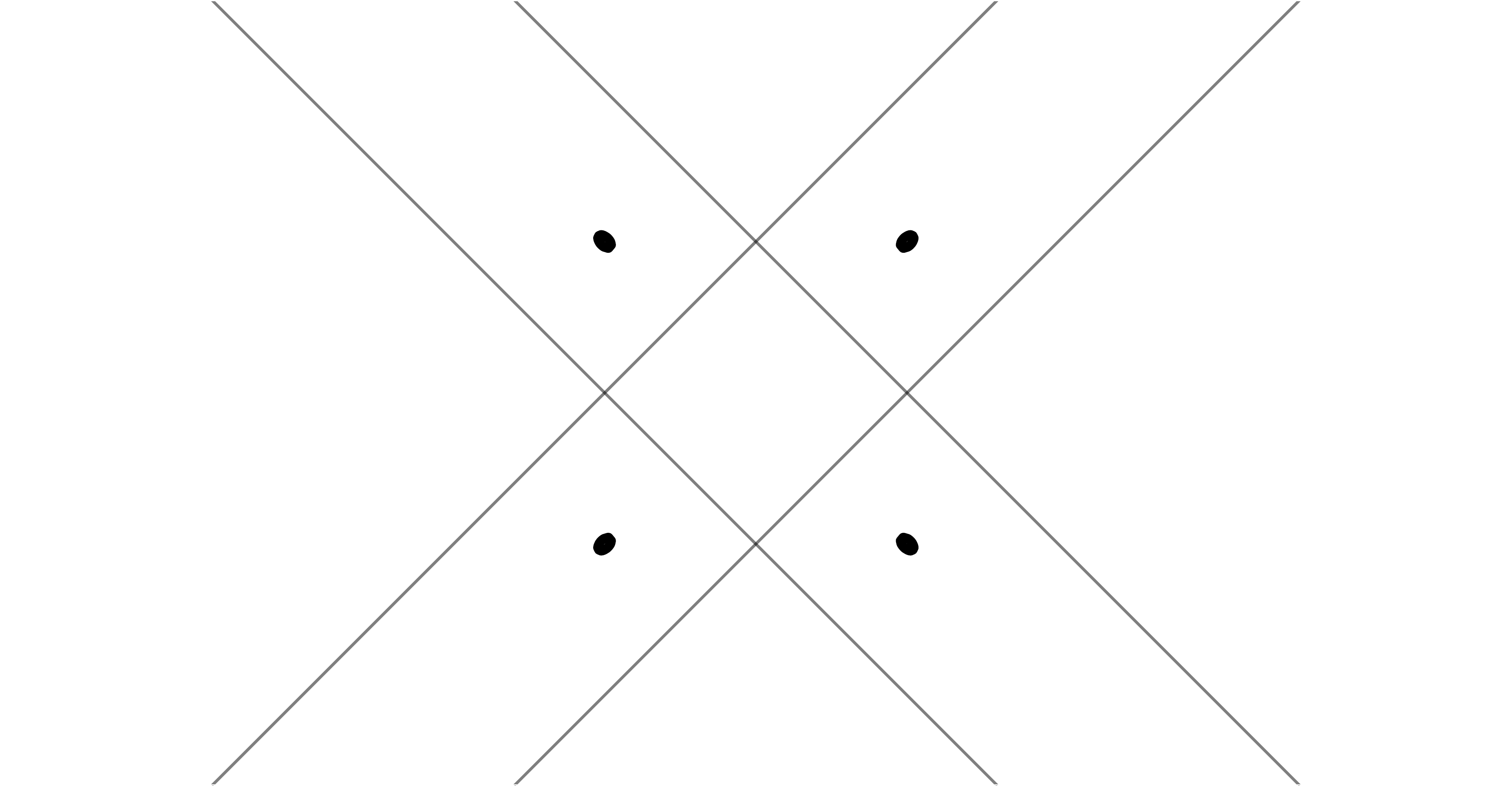} 
 }
\caption{The one-parameter family of quartic curves $C_t$}
\end{figure}

\begin{rmk}
It would be interesting to also describe the Weyl chambers of the other five components of the moduli space of smooth real plane quartic curves. A similar question can be asked for the other components of the moduli space of smooth real binary octics. For the component that corresponds to binary octics with six points real and one pair of complex conjugate points we managed to compute by hand the Coxeter diagram of this chamber. The result was already much more complicated then the diagram of the polytope $C_5$ of of Figure \ref{CoxDiag}. This leads us to believe that the Coxeter diagrams of the remaining five components of the moduli space of smooth real plane quartics will be even more complicated. Computing them would require implementing our version of Vinberg's algorithm in a computer. We expect that this will produce complicated Coxeter diagrams that do not provide much insight. 
\end{rmk}

\section*{Appendix: Involutions in Coxeter groups}\label{SectInvCox}

In this Appendix we will determine the conjugation classes of involutions in the Weyl group of type $E_7$. Weyl groups can be realized as finite Coxeter groups. The classification of conjugacy classes of involutions in a Coxeter group was done by Richardson \cite{Richardson} and Springer \cite{Springer}. Before this the classification of conjugacy classes of elements of finite Coxeter groups was obtained by Carter \cite{Carter}. We will give a brief overview of these results. 

\begin{dfn}\label{Coxeter}A Coxeter system is a pair $(W,S)$ with $W$ a group presented by a finite set of generators $S=\{s_1,\ldots,s_r\}$ subject to relations
\begin{align*}
(s_is_j)^{m_{ij}}=1 \quad \text{with} \quad 1\leq i,j \leq r
\end{align*}
where $m_{ii}=1$ and $m_{ij}=m_{ji}$ are integers $\geq 2$. We also allow $m_{ij}=\infty$ in which case there is no relation between $s_i$ and $s_j$. These relations are encoded by the Coxeter graph of $(W,S)$. This is a graph with $r$ nodes labeled by the generators. Nodes $i$ and $j$ are not connected if $m_{ij}=2$ and are connected by an edige if $m_{ij}\geq 3$ with mark $m_{ij}$ if $m_{ij} \geq 4$.
\end{dfn} 

For a Coxeter system $(W,S)$ we define an action of the group $W$ on the real vector space $V$ with basis $\{e_s\}_{s\in S}$. First we define a symmetric bilinear form $B$ on $V$ by the expression
\[ B(e_i,e_j) = -2 \cos \left(\frac{\pi}{m_{ij}}\right). \]
Then for each $s_i\in S$ the reflection: $s_i(x)=x-B(e_i,x)e_i$ preserves this form $B$. In this way we obtain a homomorphism $W\rightarrow GL(V)$ called the geometric realization of $W$. For each subset $I \subseteq S$ we can form the standard parabolic subgroup $W_I<W$ generated by the elements $\{s_i ; i \in I \}$ acting on the subspace $V_I$ generated by $\{e_i\}_{i\in I}$. We say that $W_I$ (or also $I$) satisfies the $(-1)$-condition if there is a $w_I\in W_I$ such that $w_I \cdot x = -x$ for all $x\in V_I$. The element $w_I$ necessarily equals the longest element of $(W_I,S_I)$. This implies in particular that $W_I$ is finite. Let $I,J\subseteq S$, we say that $I$ and $J$ are $W$-equivalent if there is a $w\in W$ that maps $\{e_i \}_{i \in I}$ to $\{ e_j\}_{j \in J}$. Now we can formulate the main theorem of \cite{Richardson}.

\begin{thm}[Richardson]Let $(W,S)$ be a Coxeter system and let $\mathcal{J}$ be the set of subsets of $S$ that satisfy the $(-1)$-condition. Then: 
\begin{enumerate}
\item If $c \in W$ is an involution, then $c$ is conjugate in $W$ to $w_I$ for some $I \in \mathcal{J}$.
\item Let $I,J \in \mathcal{J}$. The involutions $w_I$ and $w_J$ are conjugate in $W$ if and only if $I$ and $J$ are $W$-equivalent.
\end{enumerate}
\end{thm}

This theorem reduces the problem of finding all conjugacy classes of involutions in $W$ to finding all $W$-equivalent subsets in $S$ satisfying the $(-1)$-condition. First we determine which subsets $I\subseteq S$ satisfy the $(-1)$-condition, then we present an algorithm that determines when two subsets $I,J\subseteq S$ are $W$-equivalent. If $(W_I,S_I)$ is irreducible and satisfies the $(-1)$-condition then it is of one of the following types
\begin{align}\label{list}
A_1,B_n,D_{2n},E_7,E_8,F_4,G_2,H_3,H_4,I_2(2p)
\end{align}
with $n,p\in \mathbb{N}$ and $p\geq 4$. If $(W_I,S_I)$ is reducible and satisfies the $(-1)$-condition then $W_I$ is the direct product of irreducible, finite standard parabolic  subgroups $(W_i,S_i)$ from (\ref{list}). The Coxeter diagrams of the $(W_i,S_i)$ occur as disjoint subdiagrams of the types in the list of the diagram of $(W,S)$. The element $w_I$ is the product of the $w_{I_i}$ which act as $-1$ on the $V_{I_i}$. Now let $K \subseteq S$ be of finite type and let $w_K$ be the longest element of $(W_K,S_K)$. The element $\tau_K=-w_K$ defines a diagram involution of the Coxeter diagram of $(W_K,S_K)$ which is nontrivial if and only if $w_K \neq -1$. If $I,J\subseteq K$ are such that $\tau_K I = J$ then $I$ and $J$ are $W$-equivalent. To see this, observe that $w_Kw_I \cdot I = w_K \cdot (-I) = \tau_K I=J$. Now we define the notion of elementary equivalence.

\begin{dfn}We say that two subsets $I,J\subseteq S$ are elementary equivalent, denoted by $I \vdash J$, if $\tau_K I = J$ with $K=I\cup \{\alpha \} = J\cup \{ \beta \}$ for some $\alpha,\beta \in S$.
\end{dfn}

It is proved in \cite{Richardson} that $I$ and $J$ are $W$-equivalent if and only if they are related by a chain of elementary equivalences: $I = I_1 \vdash I_2 \vdash \ldots \vdash I_n = J$. This provides a practical algorithm to determine all the conjugation classes of involutions in a given Coxeter group $(W,S)$ using its Coxeter diagram:
\begin{enumerate}
\item Make a list of all the subdiagrams of the Coxeter diagram of $(W,S)$ that satisfy the $(-1)$-condition. These are exactly the disjoint unions of diagrams in the list (\ref{list}). Every involution in $W$ is conjugate to $w_K$ with  $K$ a subdiagram in this list. 
\item Find out which subdiagrams of a given type are $W$-equivalent by using chains of elementary equivalences.
\end{enumerate}

\begin{ex}[$E_7$]\label{invE7}We use the procedure described above to determine all conjugation classes of involutions in the Weyl group of type $E_7$. This result will be used many times later on. Since $W_7$ contains the element $-1$ the conjugation classes of involutions come in pairs $\{u,-u\}$. We label the vertices of the Coxeter diagram as in Figure \ref{labele7}

\begin{figure}
\begin{center}
\begin{tikzpicture}
  \tikzstyle{every node}=[circle,draw]  
    \node[label=60:$1$] (1) at ( 0,0) {};
    \node[label=60:$2$] (2) at ( 1,0) {};    
    \node[label=60:$3$] (3) at ( 2,0) {};
    \node[label=60:$4$] (4) at ( 3,0) {};
    \node[label=60:$7$] (7) at ( 2,1) {};
    \node[label=60:$5$] (5) at ( 4,0) {};
    \node[label=60:$6$] (6) at ( 5,0) {}; 
    \draw [-] (1) -- (2) -- (3) -- (4) -- (5) -- (6);
    \draw [-] (3) -- (7);   
\end{tikzpicture}
\caption{The labelling of the nodes of the $E_7$ diagram}
\label{labele7}
\end{center}
\end{figure}
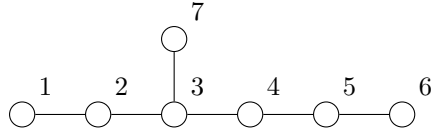

It turns out that all involutions of a given type are equivalent with the exception of type $A_1^3$. In that case there are two nonequivalent involutions as seen in Figure \ref{A13}. The types of involutions that occur are:
\begin{equation}\label{pairs}
\{ 1,E_7\} \ ,\ \{A_1,D_6\} \ ,\ \{A_1^2,D_4A_1\} \ ,\ \{A_1^3,A_1^4\} \ ,\ \{D_4,A_1^{3\prime}\}.
\end{equation}

\begin{figure}
\centering
\subfigure{
\begin{tikzpicture}
	\node[circ] (1) at (0,0) {};
	\node[blackcirc] (2) at (1,0) {};
	\node[circ] (3) at (2,0) {};
	\node[blackcirc] (4) at (3,0) {};
	\node[circ] (5) at (4,0) {};
	\node[circ] (6) at (5,0) {};
	\node[blackcirc] (7) at (2,1) {};	
	\draw [-] (1) -- (2) -- (3) -- (4) -- (5) -- (6);
	\draw [-] (3) -- (7);	
\end{tikzpicture} 
}
\subfigure{
\begin{tikzpicture}
	\node[circ] (1) at (0,0) {};
	\node[circ] (2) at (1,0) {};
	\node[circ] (3) at (2,0) {};
	\node[blackcirc] (4) at (3,0) {};
	\node[circ] (5) at (4,0) {};
	\node[blackcirc] (6) at (5,0) {};
	\node[blackcirc] (7) at (2,1) {};	
	\draw [-] (1) -- (2) -- (3) -- (4) -- (5) -- (6);
	\draw [-] (3) -- (7);	
\end{tikzpicture} 
}
\caption[test]{The involutions $A_1^3$ (left) and $A_1^{3\prime}$ (right).}
\label{A13}
\end{figure}
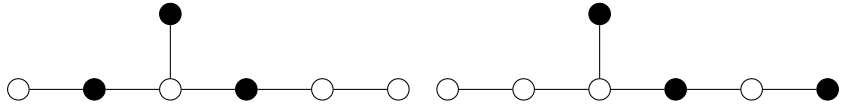

For example, consider the two subdiagrams of type $A_1$ with vertices $\{1\}$ and $\{2\}$. The diagram automorphism $\tau_{\{1,2\}}$ which is of type $A_2$ exchanges the vertices $\{1\}$ and $\{2\}$, so they are elementary equivalent. One shows in a similar way that all diagrams of type $A_1$ are equivalent. 
\end{ex}

\bibliographystyle{plain}
\bibliography{master}

\end{document}